\documentclass{article}
\usepackage{amsmath,amssymb}
\usepackage{graphicx}
\usepackage{amsthm}

\theoremstyle{definition}
\newtheorem{prop}{Proposition}
\newtheorem{thm}{Theorem}
\newtheorem{cor}{Corollary}

\begin{document}
\begin{center}
{\Large On two-parameter solutions of simultaneous ultradiscrete Painlev\'{e} II equation with parity variables}\medskip\\
Hikaru Igarashi$^*$ and Kouichi Takemura$^* $\medskip\\
$^*$Department of Mathematics, Faculty of Science and Engineering, Chuo University, 1-13-27 Kasuga, Bunkyo-ku Tokyo 112-8551, Japan\medskip\\
%$^\dag$School of Mathematics, University of Leeds, Leeds LS2 9JT, United Kingdom
\end{center}
\section*{Abstract}
We introduce a simultaneous ultradiscrete Painlev\'e II equation with parity variables, which is shown to be more suitable for studying two-parameter solutions than the single second-order ultradiscrete Painlev\'e II equation with parity variables.
We investigate several types of two-parameter solutions and the solutions which are related with the ultradiscrete limit of determinant type solutions of $q$-Painlev\'e II equation.
%\subjclass[2010]{39A13,33E17}
%\keywords{Painlev\'e equation, ultradiscrete limit, special solutions}
\section{Introduction}

The Painlev\'e equation is a non-linear ordinary differential equation of order two which has the Painlev\'e property, i.e. any movable singularity of solutions must be a pole.
The Painlev\'e equations appear frequently in the context of mathematical physics.
The second Painlev\'e equation (Painlev\'e II) is written as
\begin{equation}
\frac{d^2y}{dt^2} =2 y^3+2t y +c.
\end{equation}
In the 1990s', discrete analogues of the Painlev\'{e} equations had been found by considering alternatives of the Painlev\'e property. Ramani and Grammaticos \cite{RG} found the $q$-discrete Painlev\'e II equation ($q$-PII), which has several expression.
In this paper we adopt the expression of $q$-PII as follows:
\begin{equation}
(z(q\tau) z(\tau) + 1)(z(\tau) z(q^{-1}\tau) + 1) = \frac{a\tau^2z(\tau)}{\tau - z(\tau)}. \label{eq:qPII}
\end{equation}
Note that Painlev\'e II is recovered by the limit $q\to 1$ of $q$-PII.
Hamamoto, Kajiwara and Witte \cite{HKW} investigated special solutions of $q$-PII expressed in terms of determinants whose elements are expressed by $q$-Airy function in the case $a=q^{2N+1}$ and $N \in \mathbb{Z}$.
Another approach to investigate solutions of $q$-Painlev\'e equations is to investigate solutions in the case $q=0$.
The case $q=0$ can be realised by the ultradiscrete limit, and it is profitable to introduce a parity variable so that the ultradiscrete limit of special solutions of determinant type is considered (see \cite{IKMMS,IS,IST,IIT}).

The ultradiscrete Painlev\'e II equation with parity variables (abbreviated to p-ud PII) is described by using the parity variable $\zeta _m \in \{+1,-1\}$ and the amplitude variable $Z_ m \in {\mathbb R}$.
We assume $Q<0$.
Then p-ud PII is written as
%We assume $Q<0$ and set $a= \exp(A/\varepsilon )$, $q= \exp(Q/\varepsilon )$, $\tau =q^m$, $z(q^m ) = \zeta _m \exp (Z_m /\varepsilon )$ ($\zeta _m \in \{ \pm 1 \}$, $Z_m \in {\mathbb R}$) in Eq.(\ref{eq:qPII}).
%By letting $\varepsilon \to +0$, we obtain p-ud PII
\begin{align}
\max\Bigl[
&Z_{m+1} + 3Z_{m} + Z_{m-1} + S (\zeta_{m+1}\zeta_{m}\zeta_{m-1}), Z_{m+1} + 2Z_{m} + S (\zeta_{m+1}), \label{eq:udPIIsingle} \\
&2Z_{m} + Z_{m-1} + S (\zeta_{m-1}), Z_{m} + S (\zeta_{m}), Z_{m} + A + 2m Q + S (\zeta_{m}), \nonumber \\
&Z_{m+1} + 2Z_{m} + Z_{m-1} + mQ + S (-\zeta_{m+1}\zeta_{m-1}), \nonumber \\
&Z_{m+1} + Z_{m} + mQ + S (-\zeta_{m+1}\zeta_{m}), Z_{m} + Z_{m-1} + mQ + S (-\zeta_{m}\zeta_{m-1}) \Bigr] \nonumber \\
= \max\Bigl[
&Z_{m+1} + 3Z_{m} + Z_{m-1} + S (-\zeta_{m+1}\zeta_{m}\zeta_{m-1}), Z_{m+1} + 2Z_{m} + S (-\zeta_{m+1}), \nonumber \\
&2Z_{m} + Z_{m-1} + S (-\zeta_{m-1}), Z_{m} + S (-\zeta_{m}), Z_{m} + A + 2m Q + S (-\zeta_{m}), \nonumber \\
&Z_{m+1} + 2Z_{m} + Z_{m-1} + mQ + S (\zeta_{m+1}\zeta_{m-1}), \nonumber \\
&Z_{m+1} + Z_{m} + mQ + S (\zeta_{m+1}\zeta_{m}), Z_{m} + Z_{m-1} + mQ + S (\zeta_{m}\zeta_{m-1}), mQ \Bigr], \nonumber 
%\label{eq:udP2}
\end{align}
where the function $S:\{+1, -1\} \to \{0, -\infty\}$ is defined by 
\begin{align}
S (\omega) :=
\begin{cases}
0 & (\omega=+1), \\
-\infty & (\omega=-1).
\end{cases}
\end{align}
Note that p-ud PII is obtained from $q$-PII by setting $a= \exp(A/\varepsilon )$, $q= \exp(Q/\varepsilon )$, $\tau =q^m$ and $z(q^m ) = \zeta _m \exp (Z_m /\varepsilon )$ in Eq.(\ref{eq:qPII}) and by taking a limit $\varepsilon \to +0 $.
For details, see \cite{IKMMS,IIT}.

Isojima and the authors \cite{IIT} studied further the ultradiscrete limit of the determinant-type solutions \cite{HKW} in the case $a=q^{2N+1}$ and $N \in \mathbb{Z}_{\geq 0}$ by following the earlier studies \cite{IKMMS,IST}.
Consequently, some special solutions of p-ud PII were obtained.
An example of solutions to p-ud PII given in \cite{IIT} is described as
\begin{align}
%(\zeta^{(3)}(m), Z^{(3)}(m))=
(\zeta _m, Z_m)=
\begin{cases}
(+1,3m+21) & (m\le -19) \\
(+1,29) & (m=-18) \\
(+1,33) & (m=-17) \\
(+1,-32) & (m=-16) \\
(-1,32) & (m=-15)\\
(+1,30) & (m=-14) \\
(-1,-30) & (m=-13) \\
(+1,30) & (m=-12) \\
(+1,16) & (m=-11) \\
(-1,-16) & (m=-10) \\
(+1,-3m) & (-9 \le m \le -1)\\
((-1)^m ,0 ) & (m\geq 0),
\end{cases}
\label{eq:z3}
\end{align}
where the parameters are chosen as $Q=-3$ and $A=7Q$.
The graph of the solution in Eq.(\ref{eq:z3}) is written as Figure 1, where $\bullet $ (resp. $\circ $) represents the amplitute with $\zeta _m =+1 $ (resp. $\zeta _m =-1 $).
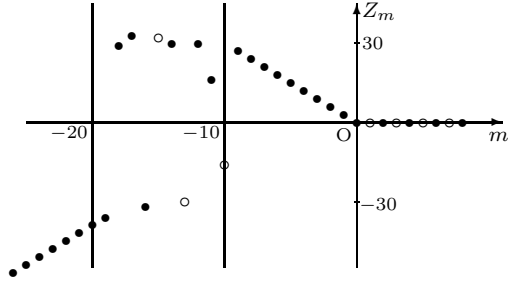
\begin{figure}
\begin{picture}(200,115)(-50,0)
\put(0,-7){\circle*{3}}
\put(5,-4){\circle*{3}}
\put(10,-1){\circle*{3}}
\put(15,2){\circle*{3}}
\put(20,5){\circle*{3}}
\put(25,8){\circle*{3}}
\put(30,11){\circle*{3}}
\put(35,14){\circle*{3}}
\put(40,79){\circle*{3}}
\put(45,83){\circle*{3}}
\put(50,18){\circle*{3}}
\put(55,82){\circle{3}}
\put(60,80){\circle*{3}}
\put(65,20){\circle{3}}
\put(70,80){\circle*{3}}
\put(75,66){\circle*{3}}
\put(80,34){\circle{3}}
\put(85,77){\circle*{3}}
\put(90,74){\circle*{3}}
\put(95,71){\circle*{3}}
\put(100,68){\circle*{3}}
\put(105,65){\circle*{3}}
\put(110,62){\circle*{3}}
\put(115,59){\circle*{3}}
\put(120,56){\circle*{3}}
\put(125,53){\circle*{3}}
\put(130,50){\circle*{3}}
\put(135,50){\circle{3}}
\put(140,50){\circle*{3}}
\put(145,50){\circle{3}}
\put(150,50){\circle*{3}}
\put(155,50){\circle{3}}
\put(160,50){\circle*{3}}
\put(165,50){\circle{3}}
\put(170,50){\circle*{3}}
\put(5,50){\vector(1,0){180}}
\put(130,-5){\vector(0,1){100}}
\put(80,-5){\line(0,1){100}}
\put(30,-5){\line(0,1){100}}
\put(132,90){{\footnotesize $Z_m$}}
\put(180,43){{\footnotesize $m$}}
\put(122,43){{\scriptsize O}}
\put(64,44){{\scriptsize $-10$}}
\put(14,44){{\scriptsize $-20$}}
\put(129,80){\line(1,0){2}}
\put(132,78){{\scriptsize $30$}}
\put(129,20){\line(1,0){2}}
\put(131,17){{\scriptsize $-30$}}
%\put(50,103){Determinant-type solution}
\end{picture}
\caption{An example of the ultradiscrete limit of a determinant-type solution}
\end{figure}

In this paper, we investigate p-ud PII (Eq.(\ref{eq:udPIIsingle})) and related equations.
It has been known that p-ud equation may not have uniqueness of solutions, and we explain it in the case of forward evolution of Eq.(\ref{eq:udPIIsingle}).
Unique evolution is assigned to the case that $(\zeta _{m+1}, Z_{m+1})$ is determined uniquely by Eq.(\ref{eq:udPIIsingle}) and the given values $(\zeta _{m-1}, Z_{m-1})$ and $(\zeta _{m}, Z_{m})$.
Conversely, indefinite evolution is assigned to the case that $(\zeta _{m+1}, Z_{m+1})$ is not determined uniquely by them.
Here we examine the uniqueness and indefiniteness by using the solution in Eq.(\ref{eq:z3}).
Let us investigate forward evolution with the initial values $(\zeta _{-17}, Z_{-17})= (+1,33)$ and $(\zeta _{-16}, Z_{-16})= (+1,-32)$.
We substitute $m=-16$ into Eq.(\ref{eq:udPIIsingle}), and we have 
\begin{align}
\max [ Z_{-15} -63 + S (\zeta_{-15} ),43, Z_{-15} + 17 + S (-\zeta_{-15}) ] \nonumber \\
= \max [Z_{-15} -63 + S (-\zeta_{-15} ), Z_{-15} + 17 + S (\zeta_{-15} ),49 ] . 
\end{align}
If $\zeta _{-15} =+1$, then we have $\max [ Z_{-15} -63,43] = \max [ Z_{-15} + 17,49 ]$ and it turns out that there is no solution.
Hence we have $\zeta _{-15} =-1$, $\max [ Z_{-15} +17,43] = \max [ Z_{-15} -63,49 ]$ and $Z_{-15} = 32$.
Therefore the evolution to $(\zeta _{-15}, Z_{-15}) = (-1,32)$ is unique.
Next we substitute $m=-15$ into Eq.(\ref{eq:udPIIsingle}). Then we have 
\begin{align}
\max [  Z_{-14} + 77 + S (-\zeta_{-14}), Z_{-14} + 77 + S (\zeta_{-14}), 45 ] \nonumber \\
= \max [ 101 , Z_{-14} + 77 + S (\zeta_{-14}), Z_{-14} + 77 + S (-\zeta_{-14}) ] .
\end{align}
If $Z_{-14} \geq 24$, then the ultradiscrete equation is satisfied.
Thus the evolution to $(\zeta _{-14}, Z_{-14}) $ is indefinite.

In the general setting, the condition that forward evolution in Eq.(\ref{eq:udPIIsingle}) is unique is determined in Proposition \ref{prop:Zsing3}.
Moreover, if the evolution is unique, then the amplitude function $Z_{m+1} $ is written in a simpler form 
\begin{equation}
Z_{m +1} = -Z_m +\max [0, A+2mQ -\max[0, mQ-Z_m ] -\max[0,Z_{m-1} +Z_{m}]] .
\end{equation}
However it is shown that Eq.(\ref{eq:udPIIsingle}) may not govern the solution for all $m \in \mathbb{Z}$.
Namely, there exists no solution to single p-ud PII such that any forward and backward evolution are unique for all $m \in \mathbb{Z}$ (see Theorem \ref{thm:nonentire}).

In order to avoid indefinite evolution, we introduce another variable to p-ud PII.
We set $\tau = q^m $ and $y(q^{m+1}) = z(q^{m+1}) z(q^{m}) + 1$ in Eq.(\ref{eq:qPII}).
Then Eq.(\ref{eq:qPII}) is written as the simultaneous equation
\begin{align}
& y(q^{m+1}) y(q^m) = \dfrac{a q^{2m}z(q^m)}{q^m - z(q^m)} , \quad  y(q^{m+1}) = z(q^{m+1}) z(q^m) + 1  . \label{eq:qPIIyzqm}
\end{align}
Note that the introduction of the variable $y(q^{m})$ is essentially due to Murata \cite{Mu} for ultradiscrete Painlev\'e II equation without parity variables.
Let $(\eta _m, Y_m)$ be the parity variable and the amplitude function determined by $y(q^m)= \eta _m \exp (Y_m /\varepsilon )$.
Then the corresponding p-ud equation is written as
\begin{align}
\max[ & mQ-Z_m +Y_{m+1}+Y_m+ S(\eta_{m+1}\eta_{m}) , \label{eq:udYYZ} \\
& Y_{m+1}+Y_m + S(-\zeta_{m}\eta_{m+1}\eta_{m}) , A+2m Q+ S(-\zeta_{m}) ] \nonumber\\
 = \max[ & mQ-Z_m +Y_{m+1}+Y_m + S(-\eta_{m+1}\eta_{m}) , \nonumber\\
& Y_{m+1}+Y_m + S(\zeta_{m}\eta_{m+1}\eta_{m}) , A+2m Q+ S(\zeta_{m}) ] , \nonumber \\
\max[ & Y_{m+1} +S(\eta_{m+1}) , Z_m+Z_{m+1}+S( -\zeta_{m+1}\zeta_{m}) ] \label{eq:udYZZ} \\
 = \max[ & Y_{m+1}+S(-\eta_{m+1}) , Z_m+Z_{m+1}+S( \zeta_{m+1}\zeta_{m}) ,0] , \nonumber 
\end{align}
which we will show in section \ref{sec:pudyz}.
The condition that the evolution by the simultaneous equations (\ref{eq:udYYZ}) and (\ref{eq:udYZZ}) is unique is written as $(\zeta _m , Z_m) \neq (+1,mQ)$ and $(\eta _{m+1} ,Y_{m+1}) \neq (+1,0)$ (see Corollary \ref{cor:uniindef}).
Moreover, if $(\zeta _m , Z_m) \neq (+1,mQ)$ and $(\eta _{m+1} ,Y_{m+1}) \neq (+1,0)$, then the equations are written in a simpler form, i.e. the amplitude functions satisfy
\begin{align}
Y_{m+1} + Y_m &= A+ 2m Q- \max [ mQ-Z_m, 0 ] , \label{eq:evolu2} \\
Z_{m+1} + Z_m &= \max[Y_{m+1}, 0] , \label{eq:evolu1} 
\end{align}
and the parity functions satisfy
\begin{equation}
\zeta_{m+1}\zeta_{m} =\left\{ 
\begin{array}{cc}
\eta _{m+1} & (Y_{m+1} >0 ),\\
-1 & ( Y_{m+1} \leq 0), 
\end{array}
\right.
 \; \eta_{m+1}\eta_{m} =\left\{ 
\begin{array}{cc}
\zeta _{m} & (Z_{m} <mQ ), \\
-1 & (Z_{m} \geq mQ ). 
\end{array}
\right. 
\label{eq:evolp}
\end{equation}
In section \ref{sec:etaYIIT}, we investigate the value of $(\eta _{m}, Y_{m}) $ for the ultradiscrete function $(\zeta _{m}, Z_{m}) $ obtained in \cite{IIT} and clarify a relationship with indefinite evolution.
We now explain it by the example in Eq.(\ref{eq:z3}).
We choose the initial values $(\zeta _{-17}, Z_{-17})= (+1,33)$ and $(\zeta _{-16}, Z_{-16})= (+1,-32)$ as the single equation for $(\zeta _{m}, Z_{m}) $.
We introduce the values of $(\eta _{m}, Y_{m}) $ which satisfy the simultaneous equation.
It follows from $Z_{-16} + Z_{-17} = \max[Y_{-16}, 0]$ that $Y_{-16} =1$ and also follows from $\zeta_{-16}\zeta_{-17} = \eta _{-16}$ that $\eta_{-16} = +1$.
By the evolution determined by Eqs.(\ref{eq:udYYZ}) and (\ref{eq:udYZZ}), 
the values $( \eta_ m , Y_m ) $ are determined for $-18 \leq m \leq -9$ as
\begin{align}
( \eta_ m , Y_m )=
\begin{cases}
(+1,0) & (m=-18) \\
(+1,62) & (m=-17) \\
(+1,1) & (m=-16) \\
(+1,-6) & (m=-15)\\
(-1,62) & (m=-14) \\
(-1,-11) & (m=-13) \\
(+1,-1) & (m=-12) \\
(+1,46) & (m=-11) \\
(+1,-18) & (m=-10) \\
(-1,11) & (m=-9 )
\end{cases}
\label{eq:y3}
\end{align}
and the values $(\zeta _{m}, Z_{m})$ for $-18 \leq m \leq -9$ coincide with the ones in Eq.(\ref{eq:z3}).
In the example, we have unique evolution for $-18 \leq m \leq -9$, although indefinite evolution occurs by the values $( \eta_{-18} , Y_{-18} )=  (+1,0)$ and $(\zeta_{-9} , Z_{-9} )=(+1,27)$.

In the master's thesis of the first author \cite{Iga}, p-ud limit $(\zeta _{m},Z_{m})  $ of the determinant-type solutions for the case $a=q^{2M+1}$ and $M \in \mathbb{Z}_{< 0}$ was investigated. 
We review the results of the p-ud limit and also investigate the value of $(\eta _{m}, Y_{m}) $.

In section \ref{sec:twoparameter}, we investigate two parameter solutions of p-ud PII.
Note that Murata \cite{Mu} had investigated two parameter solutions of ultradiscrete PII without parity variables, and our solutions include the patterns which did not appear in \cite{Mu}.

In section \ref{sec:perturb}, we apply the two parameter solutions to obtain the ones which is perturbed from the solutions in section \ref{sec:etaYIIT}.
As a perturbation of the solution in Eqs.(\ref{eq:z3}) and (\ref{eq:y3}), we investigate the solution whose initial value is given by $( \eta_{-18} , Y_{-18} )= (+1,0- \varepsilon)$ and $(\zeta _{-18}, Z_{-18} )= (+1,29)$ $(0<4 \varepsilon <1 )$. 
Then the indefiniteness of the solution disappears and a graph of the solution $(\zeta _{m},Z_{m})  $  is written as Figure 2.
See Eq.(\ref{eq:pertex1}) for the exact values.
Hence, if $m \leq -19  $ or $m \geq -8$, then the values of $(\zeta _m, Z_m)$ are completely different from Eq.(\ref{eq:z3}) (or Figure 1).
In other words, the p-ud limits of determinant-type solutions are not stable under the perturbation of initial values.
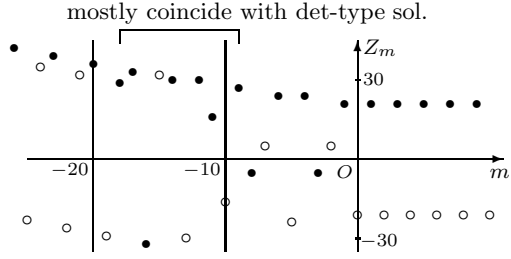
\begin{figure}
\begin{picture}(200,110)(-50,20)
\put(20,92){\circle*{3}}
\put(25,27){\circle{3}}
\put(30,85){\circle{3}}
\put(35,89){\circle*{3}}
\put(40,24){\circle{3}}
\put(45,82){\circle{3}}
\put(50,86){\circle*{3}}
\put(55,21){\circle{3}}
\put(60,79){\circle*{3}}
\put(65,83){\circle*{3}}
\put(70,18){\circle*{3}}
\put(75,82){\circle{3}}
\put(80,80){\circle*{3}}
\put(85,20){\circle{3}}
\put(90,80){\circle*{3}}
\put(95,66){\circle*{3}}
\put(100,34){\circle{3}}
\put(105,77){\circle*{3}}
\put(110,45){\circle*{3}}
\put(115,55){\circle{3}}
\put(120,74){\circle*{3}}
\put(125,26){\circle{3}}
\put(130,74){\circle*{3}}
\put(135,45){\circle*{3}}
\put(140,55){\circle{3}}
\put(145,71){\circle*{3}}
\put(150,29){\circle{3}}
\put(155,71){\circle*{3}}
\put(160,29){\circle{3}}
\put(165,71){\circle*{3}}
\put(170,29){\circle{3}}
\put(175,71){\circle*{3}}
\put(180,29){\circle{3}}
\put(185,71){\circle*{3}}
\put(190,29){\circle{3}}
\put(195,71){\circle*{3}}
\put(200,29){\circle{3}}
%\put(185,71){\circle*{3}}
%\put(190,29){\circle{3}}
\put(25,50){\vector(1,0){180}}
\put(150,15){\vector(0,1){80}}
\put(100,15){\line(0,1){80}}
\put(50,15){\line(0,1){80}}
\put(152,90){{\footnotesize $Z_m$}}
\put(200,43){{\footnotesize $m$}}
\put(142,43){{\scriptsize $O$}}
\put(84,44){{\scriptsize $-10$}}
\put(34,44){{\scriptsize $-20$}}
\put(149,80){\line(1,0){2}}
\put(152,78){{\scriptsize $30$}}
\put(149,20){\line(1,0){2}}
\put(151,17){{\scriptsize $-30$}}
%\put(40,103){Perturbed solution}
\put(40,103){{\small mostly coincide with det-type sol.}}
\put(60,99){\line(0,-1){5}}
\put(105,99){\line(0,-1){5}}
\put(60,99){\line(1,0){45}}
%\put(55,9){\line(0,1){5}}
%\put(55,9){\line(-1,0){40}}
%\put(20,0){{\small type $--$}}
%\put(57,9){\line(0,1){5}}
%\put(68,9){\line(0,1){5}}
%\put(77,9){\line(0,1){5}}
%\put(105,9){\line(0,1){5}}
%\put(57,9){\line(1,0){11}}
%\put(77,9){\line(1,0){28}}
%\put(65,0){{\small type $-A$}}
%\put(145,9){\line(0,1){5}}
%\put(145,9){\line(1,0){35}}
%\put(145,0){{\small type $+A$}}
%\put(180,9){\line(0,1){5}}
%\put(185,9){\line(0,1){5}}
%\put(185,9){\line(1,0){30}}
%\put(190,0){{\small type $++$}}
\end{picture}
\caption{A perturbed solution}
\end{figure}

This paper is organized as follows.  
In section \ref{sec:pudyz}, we derive the simultaneous equation (Eqs.(\ref{eq:udYYZ}) and (\ref{eq:udYZZ})) of p-ud PII and investigate some properties of the simultaneous p-ud PII and the single p-ud PII.
In section \ref{sec:etaYIIT}, we investigate the values of $(\eta _{m}, Y_{m}) $ for the ultradiscrete function $(\zeta _{m}, Z_{m}) $ obtained in \cite{Iga,IIT} and clarify a relationship with indefinite evolution.
In section \ref{sec:twoparameter}, we investigate two parameter solutions of p-ud PII.
In section \ref{sec:perturb}, we obtain the solutions which are perturbed from the ones in section \ref{sec:etaYIIT} by using two parameter solutions.
% and \ref{sec:udsolIga}.
In section \ref{sec:concl}, we give concluding remarks.
In the appendix, we review a procedure of obtaining the p-ud limit $(\zeta _{m},Z_{m})$ of the determinant-type solutions in the case $a=q^{2M+1}$ and $M \in \mathbb{Z}_{< 0}$, which is based on \cite{Iga}.
Throughout this paper, we assume $Q<0$.

\section{Simultaneous ultradiscrete equation with parity variables} \label{sec:pudyz}
We rewrite the simultaneous equation of $q$-PII (Eq.(\ref{eq:qPIIyzqm})) as
\begin{align}
& y_{m+1} y_m  = \dfrac{a q^{2m}z_m}{q^m - z_m} , \quad y_{m+1} = z_{m+1} z_m + 1 . \label{eq:qPIIyzm}
\end{align}
We fix a value $Q(<0)$ and assume $0<q<1$.
Introduce a variable $\varepsilon >0$ by $q= \exp(Q/\varepsilon ) $ and write $a= \exp(A/\varepsilon ) $.
We assume that there exists a one-parameter family of the solution ($y_m (\varepsilon )$, $z_m(\epsilon ) $) $(\varepsilon $: positive and sufficielntly small$)$ such that 
\begin{equation}
 y_m (\varepsilon ) = \eta _m \exp (Y_m (\varepsilon )/\varepsilon ), \; z_m (\varepsilon ) = \zeta _m \exp (Z_m (\varepsilon )/\varepsilon )
\label{eq:qyzepsilon}
\end{equation}
and the limits $Y_m (\varepsilon ) \to Y_m$ and $Z_m (\varepsilon ) \to Z_m$ exist as $\varepsilon \to +0$, where $\eta _m $ and $\zeta _m $ represent the signs of $y_m (\varepsilon )$ and $z_m(\epsilon ) $.
Then we call $(\eta _m , Y_m)$ (resp. $(\zeta _m , Z_m)$) the p-ud analogue of $y_m (\varepsilon ) $ (resp. $z_m (\varepsilon )$).
On the first equation of Eq.(\ref{eq:qPIIyzm}), we multiply the denominator of the right hand side and substitute Eq.(\ref{eq:qyzepsilon}) into it.
We apply the formulas as
\begin{equation}
\eta _{m+1} \eta _m \zeta _m  = \exp (S(\eta _{m+1} \eta _m \zeta _m )/\varepsilon ) - \exp (S(-\eta _{m+1} \eta _m \zeta _m )/\varepsilon ) \quad  (\varepsilon >0),
\end{equation}
transpose the negative terms to the other side of equality, and multiply $e^{-Z_{m}(\varepsilon ) /\varepsilon} $.
Then we have
\begin{align}
& e^{(Y_{m+1}(\varepsilon ) +Y_{m}(\varepsilon ) +mQ -Z_{m}(\varepsilon ) + S(\eta _{m+1} \eta _m )) /\varepsilon} +e^{(A+2mQ + S(-\zeta _m )) /\varepsilon} \\
&  + e^{(Y_{m+1}(\varepsilon ) +Y_{m}(\varepsilon ) + S(-\eta _{m+1} \eta _m \zeta _m)) /\varepsilon} = e^{(Y_{m+1}(\varepsilon ) +Y_{m}(\varepsilon ) +mQ -Z_{m}(\varepsilon ) + S(-\eta _{m+1} \eta _m )) /\varepsilon} \nonumber \\
& + e^{(A + 2mQ + S(\zeta _m )) /\varepsilon}+ e^{(Y_{m+1}(\varepsilon ) +Y_{m}(\varepsilon ) + S(\eta _{m+1} \eta _m \zeta _m)) /\varepsilon} . \nonumber
\end{align}
It is easy to show that if the limits $X_j(\varepsilon) \to X_j $ $(\varepsilon \to +0)$ exist for $j=1,\dots ,n $, then
\begin{equation}
\lim_{\varepsilon \to +0} \varepsilon \log (e^{X_1(\varepsilon) /\varepsilon} +\dots + e^{X_n(\varepsilon) /\varepsilon}) = \max [X_1, \dots, X_n] . \label{eq:limit_furmula}
\end{equation}
Therefore we have the following ultradiscrete equation with parity variables:
\begin{align}
 \max[ & mQ-Z_m+Y_{m+1}+Y_m+ S(\eta_{m+1}\eta_{m}) , \label{eq:YZ2} \\
& Y_{m+1}+Y_m + S(-\zeta_{m}\eta_{m+1}\eta_{m}) , A+2m Q+ S(-\zeta_{m}) ] \nonumber \\
 = \max[ &  mQ-Z_m +Y_{m+1}+Y_m+ S(-\eta_{m+1}\eta_{m})  ,\nonumber\\
&   Y_{m+1}+Y_m + S(\zeta_{m}\eta_{m+1}\eta_{m}), A+2m Q+ S(\zeta_{m}) ], \nonumber
\end{align}
which coincides with Eq.(\ref{eq:udYYZ}). 
We also obtain Eq.(\ref{eq:udYZZ}) from the second equation of Eq.(\ref{eq:qPIIyzm}).
It follows from the construction of the p-ud equation that if ($y_m $, $z_m $) is a solution of Eq.(\ref{eq:qPIIyzm}) and $(\eta _m , Y_m)$ and $(\zeta _m , Z_m)$ are the p-ud analogue of $y_m $ and $z_m $, then $(\eta _m , Y_m)$ and $(\zeta _m , Z_m)$ satisfy Eqs.(\ref{eq:udYYZ}) and (\ref{eq:udYZZ}).
Note that if $y(m)$ is expanded into series of $q=e^{Q/\varepsilon}$ as
\begin{align}
& y(m) = \hat{\eta } (m)  q^{\hat{Y}(m)}\sum_{k=0}^{\infty}d(k,m)q^k, \label{eq:expand} 
\end{align}
where $\hat{\eta } (m) \in \{ \pm 1 \} $ and $d(0,m)>0$, then the p-ud analogue of $y(m )$ is $(\hat{\eta } (m)  , \hat{Y}(m) Q)$.

We investigate uniqueness and indefiniteness of Eqs.(\ref{eq:udYYZ}) and (\ref{eq:udYZZ}), and rewrite the equations into a simpler form.
\begin{prop} \label{prop:uniindef}
(i) If $(\zeta _{m} ,Z_m )\neq (+1 , mQ)$, then Eq.(\ref{eq:udYYZ}) is equivalent to
\begin{equation}
Y_{m+1} + Y_m = A+ 2m Q- \max [ mQ-Z_m, 0 ] 
\label{eq:evoluY}
\end{equation}
and 
\begin{equation}
\eta_{m+1}\eta_{m} =\left\{ 
\begin{array}{cc}
\zeta _{m} & (Z_{m} <mQ ),\\
-1 & (Z_{m} \geq mQ ) .
\end{array}
\right.
\end{equation}
(ii) If $(\zeta _{m} ,Z_m ) = (+1 , mQ)$, then Eq.(\ref{eq:udYYZ}) is equivalent to $Y_m+Y_{m+1} \geq A+2mQ$. \\
(iii) If $(\eta _{m+1},Y_{m+1} ) \neq (+1,0)$, then Eq.(\ref{eq:udYZZ}) is equivalent to
\begin{equation}
Z_{m}+Z_{m+1}=\max [Y_{m+1},0]
\end{equation}
and
\begin{equation}
\zeta_{m+1}\zeta_{m} =\left\{ 
\begin{array}{cc}
\eta _{m+1} & (Y_{m+1} >0 ),\\
-1 & ( Y_{m+1} \leq 0) .
\end{array}
\right.
\end{equation}
(iv) If $(\eta _{m+1},Y_{m+1} ) = (+1,0)$, then Eq.(\ref{eq:udYZZ}) is equivalent to $Z_m+Z_{m+1} \leq 0$.
\end{prop}
\begin{proof}
We show (i) and (ii).
If $\zeta _{m}=-1$, then it follows from Eq.(\ref{eq:udYYZ}) that
\begin{align}
\max[ & mQ-Z_m +Y_{m+1} +Y_m + S(\eta_{m+1}\eta_{m}) , \\
& Y_{m+1} +Y_m  + S(\eta_{m+1}\eta_{m}) , A +2m Q ]  \nonumber \\
 = \max[& mQ-Z_m +Y_{m+1} +Y_m + S(-\eta_{m+1}\eta_{m}) , Y_{m+1} +Y_m  + S(-\eta_{m+1}\eta_{m}) ]. \nonumber 
\end{align}
Therefore we have $\eta_{m+1}\eta_{m} =-1$ and 
\begin{align}
 \max[ mQ-Z_m +Y_{m+1} +Y_m , Y_{m+1} +Y_m   ] = A+2m Q ,
\end{align}
which is equivalent to Eq.(\ref{eq:evoluY}).

If $\zeta _{m}=+1$, then it follows from Eq.(\ref{eq:udYYZ}) that
\begin{align}
\max[ & mQ-Z_m +Y_{m+1} +Y_m + S(\eta_{m+1}\eta_{m}) , Y_{m+1} +Y_m + S(-\eta_{m+1}\eta_{m}) ]  \nonumber \\
 = \max[ & mQ-Z_m +Y_{m+1} +Y_m + S(-\eta_{m+1}\eta_{m}) ,\\
&  Y_{m+1} +Y_m + S(\eta_{m+1}\eta_{m}) ,A+2m Q ]. \nonumber 
\end{align}
If $mQ-Z_m >0$ (resp. $mQ-Z_m <0 $), then we have $\eta_{m+1}\eta_{m} =+1$ (resp. $\eta_{m+1}\eta_{m} =-1$) and $mQ-Z_m +Y_{m+1} +Y_m  = A +2m Q $ (resp. $Y_{m+1} +Y_m = A+2m Q $).
If $mQ-Z_m =0$, then the equation is written as $Y_{m+1} +Y_m = \max[ Y_{m+1} +Y_m , A +2m Q ] $, which is equivalent to $Y_{m+1} +Y_m \geq A+2m Q $.
Therefore we have (i) and (ii).

(iii) and (iv) are shown similarly.
\end{proof}
Therefore, if $(\zeta _m , Z_m) \neq (+1,mQ)$ and $(\eta _{m+1} ,Y_{m+1}) \neq (+1,0)$, then the amplitude function satisfies Eqs.(\ref{eq:evolu2}) and (\ref{eq:evolu1}).

Assume that the values $(\eta _{m_0}, Y_{m_0})$ and $(\zeta _{m_0}, Z_{m_0} )$ are given.
On the forward evolution, $(\eta _{m_0+1}, Y_{m_0 +1})$ may be determined by Eq.(\ref{eq:udYYZ}) and $(\zeta _{m_0+1}, Z_{m_0 +1})$ may be determined by Eq.(\ref{eq:udYZZ}).
On the backward evolution, $(\zeta _{m_0-1}, Z_{m_0 -1})$ may be determined by Eq.(\ref{eq:udYZZ}) and $(\eta _{m_0-1}, Y_{m_0 -1})$ may be determined by Eq.(\ref{eq:udYYZ}).
Uniqueness and indefiniteness for the time evolution readily follow from Proposition \ref{prop:uniindef}. 
\begin{cor} \label{cor:uniindef}
(i) If $(\eta _{m},Y_{m} ) \neq (+1,0)$ (resp. $(\eta _{m},Y_{m} ) = (+1,0)$), then the forward evolution to determine $(\zeta _{m},Z_{m} )$ is unique (resp. indefinite).\\
(ii) If $(\zeta _{m} ,Z_m )\neq (+1 , mQ)$ (resp. $(\zeta _{m} ,Z_m ) = (+1 , mQ)$), then the forward evolution to determine $(\eta _{m+1},Y_{m+1} )$ is unique (resp. indefinite).\\
(iii) If $(\zeta _{m} ,Z_m )\neq (+1 , mQ)$ (resp. $(\zeta _{m} ,Z_m ) = (+1 , mQ)$), then the backward evolution to determine $(\eta _{m},Y_{m} )$ is unique (resp. indefinite).\\
(iv) If $(\eta _{m},Y_{m} ) \neq (+1,0)$ (resp. $(\eta _{m},Y_{m} ) = (+1,0)$), then the backward evolution to determine $(\zeta _{m-1},Z_{m-1} )$ is unique (resp. indefinite).
\end{cor}

We now compare the single p-ud PII (Eq.(\ref{eq:udPIIsingle})) with the simultaneous p-ud PII (Eqs.(\ref{eq:udYYZ}) and (\ref{eq:udYZZ})).
We investigate uniqueness and indefiniteness of the forward evolution of the single p-ud PII.
\begin{prop} \label{prop:singleudPIIevol}
Set
\begin{align}
& \tilde{Z} = A + 2mQ - \max[Z_{m -1} + Z_m , 0] - \max[mQ - Z_m , 0], \label{eq:XYX'} \\
& Z' = -Z_m + \max[\tilde{Z} , 0]. \nonumber
\end{align}
Then the forward evolution of Eq.(\ref{eq:udPIIsingle}) is described as follows:\\
(i) Assume that $(\zeta_{m -1}, \zeta_ m) = (+1, +1)$.
If $mQ - Z_m > 0$ and $\tilde{Z} \neq 0$, then $(\zeta_{m +1}, Z_{m +1}) = (\zeta_m \mbox{sgn} \tilde{Z}, Z') $.
If $mQ - Z_m < 0$, then $(\zeta_{m +1}, Z_{m +1}) = (-\zeta_m , Z') $.
If $mQ-Z_m=0$ or ($mQ - Z_m > 0$ and $\tilde{Z} = 0$), then the evolution is indefinite. \\
(ii) Assume that $(\zeta_{m -1}, \zeta_ m) = (+1, -1)$.
If $Z_{m -1} + Z_m > 0$ and $\tilde{Z} \neq 0$, then $(\zeta_{m +1}, Z_{m +1}) = (\zeta_m \mbox{sgn} \tilde{Z}, Z') $.
If $Z_{m -1} + Z_m < 0$, then $(\zeta_{m +1}, Z_{m +1}) = (-\zeta_m , Z') $.
If $Z_{m -1} + Z_m=0$ or ($Z_{m -1} + Z_m > 0$ and $\tilde{Z} = 0$), then the evolution is indefinite.\\
(iii) Assume that $(\zeta_{m -1}, \zeta_ m) = (-1, +1)$.
If $(mQ - Z_m )(Z_{m -1} + Z_m) < 0$ and $\tilde{Z} \neq 0$, then $(\zeta_{m +1}, Z_{m +1}) = (\zeta_m \mbox{sgn} \tilde{Z}, Z') $.
If $(mQ - Z_m )(Z_{m -1} + Z_m) > 0$, then $(\zeta_{m +1}, Z_{m +1}) = (-\zeta_m , Z') $.
If $(mQ - Z_m )(Z_{m -1} + Z_m)=0$ or ($(mQ - Z_m )(Z_{m -1} + Z_m) < 0$ and $\tilde{Z} = 0$), then the evolution is indefinite.\\
(iv) Assume that $(\zeta_{m -1}, \zeta_ m) = (-1, -1)$.
Then we have $(\zeta_{m +1}, Z_{m +1}) = (-\zeta_m , Z') $.
\end{prop} \label{prop:Zsing1}
\begin{proof}
We show (i).
Assume that $(\zeta_{m -1}, \zeta_ m) = (+1, +1)$. If $mQ-Z_m>0$, then Eq.(\ref{eq:udPIIsingle}) is written as
\begin{align}
\max [& Z_{m +1} + Z_m +\max[ Z_{m -1} + Z_m,0 ] + S(\zeta_{m +1}), \\
& Z_{m +1} + mQ + \max [ Z_{m -1} + Z_m  , 0 ] + S(-\zeta_{m +1}), A + 2mQ ] \nonumber \\
= \max [& Z_{m +1} + Z_m +\max[ Z_{m -1} + Z_m,0 ] + S(-\zeta_{m +1}),\nonumber \\
& Z_{m +1} + mQ + \max [ Z_{m -1} + Z_m  , 0 ] + S(\zeta_{m +1}), \nonumber \\
&  \max[ Z_{m -1} + Z_m ,0] + mQ- Z_m] . \nonumber
\end{align}
If $ A + 2mQ >  \max[ Z_{m -1} + Z_m ,0] + mQ- Z_m$ (i.e. $\tilde{Z} > 0$), then we have $\zeta_{m +1} =+1 $ and $Z_{m +1} = A + mQ - \max [ Z_{m -1} + Z_m  , 0 ] $, which is equivalent to $Z_{m +1} = Z'$.
If $ A + 2mQ < \max[ Z_{m -1} + Z_m ,0] + mQ- Z_m$ (i.e. $\tilde{Z} < 0$), then we have $\zeta_{m +1} =-1 $ and $Z_{m +1} = - Z_m$, which is also equivalent to $Z_{m +1} = Z'$.
If $ A + 2mQ = \max[ Z_{m -1} + Z_m ,0] + mQ- Z_m$ (i.e. $\tilde{Z} = 0$), then we have $Z_{m +1} \leq A + mQ - \max [ Z_{m -1} + Z_m  , 0 ] $, which implies that the evolution is indefinite.

If $mQ-Z_m<0$, then Eq.(\ref{eq:udPIIsingle}) is written as
\begin{align}
 \max [ & \max[ Z_{m -1} + Z_m,0 ] + \max[ S(\zeta_{m +1}), mQ -Z_m  + S(-\zeta_{m +1}) ]\\
& + Z_{m +1} + Z_m , Z_{m -1} + Z_m, 0, A + 2mQ ] \nonumber \\
= \max[ & Z_{m -1} + Z_m,0 ]  + \max [ S(-\zeta_{m +1}) ,mQ -Z_m + S(\zeta_{m +1})] + Z_{m +1} + Z_m . \nonumber 
\end{align}
Hence we have $\zeta_{m +1} =-1$ and $Z_{m +1} = - Z_m - \max[ Z_{m -1} + Z_m,0 ] + \max[Z_{m -1} + Z_m, 0, A + 2mQ  ]= -Z_m + \max[\tilde{Z} , 0]$.

If $mQ-Z_m =0$, then Eq.(\ref{eq:udPIIsingle}) is written as
\begin{align}
\max [& Z_{m +1} + mQ + \max [ Z_{m -1} + mQ ,0],  Z_{m -1} + mQ, 0, A + 2mQ ] \nonumber \\
= \max [& Z_{m +1} + mQ + \max [ Z_{m -1} + mQ ,0 ], Z_{m -1} + mQ, 0],
\end{align}
and it holds if the value $Z_{m+1} $ is large enough.
Thus the evolution is indefinite in this case.
%If $A+2mQ \leq \max[Z_{m -1} + mQ, 0 ]$ (resp. $A+2mQ > \max[Z_{m -1} + mQ, 0 ]$), then the equation holds for all $Z_{m+1}$ (resp. for under the condition $Z_{m +1} + mQ + \max [ Z_{m -1} + mQ ,0] \geq A + 2mQ $.

Summarizing the above, we obtain (i).
(ii), (iii) and (iv) are shown similarly.
\end{proof}
By arranging the previous proposition, we have
\begin{prop} \label{prop:Zsing3}
If $(\zeta _{m -1} \zeta _{m}, Z_{m -1} +Z_{m} ) = (-1, 0)$, $(\zeta _{m}, Z_{m} ) = (+1, mQ)$, ($(\zeta _{m}, Z_{m} ) = (+1, -A-mQ)$, $Z_{m -1} + Z_m < 0$ and $mQ - Z_m > 0$), ($(\zeta _{m-1}, Z_{m-1} ) = (+1, A+mQ)$, $Z_{m -1} + Z_m > 0$ and $mQ - Z_m > 0$) or ($(\zeta _{m-1} \zeta _{m}, Z_{m-1} +Z_{m} ) = (-1, A+2mQ)$, $Z_{m -1} + Z_m > 0$ and $mQ - Z_m < 0$), then the forward evolution by Eq.(\ref{eq:udPIIsingle}) is indefinite.
Otherwise, the forward evolution is unique, and the amplitude function $Z_{m+1} $ is written as
\begin{equation}
Z_{m +1} = -Z_m +\max [0, A+2mQ -\max[0, mQ-Z_m ] -\max[0,Z_{m-1} +Z_{m}]] .
\label{eq:Zm+1ue}
\end{equation}
The sign function $\zeta _{m+1}$ is written as $\zeta _{m+1} = -\zeta _m $ or $\zeta _{m+1} =\zeta _m \mbox{sgn} \tilde{Z}$, 
%\begin{equation}
%\zeta _{m+1} = \left\{
%\begin{array}{l}
%-\zeta _m \\
%\zeta _m \mbox{sgn} \tilde{Z}
%\end{array}
%\right.
%\label{eq:zetam+1ue}
%\end{equation}
whose conditions are described in Proposition \ref{prop:singleudPIIevol}.
\end{prop}
Note that $q$-PII is written as 
\begin{align}
 \zeta_{m + 1}  e^{Z_{m+1} /\varepsilon }= & \zeta_{m}  e^{-Z_{m} /\varepsilon } \cdot \label{eq:qPIIsignamp} \\
& \Big\{ \frac{e^{(A+2 mQ )/\varepsilon} }{(\zeta_{m - 1} \zeta_{m}e^{(Z_{m-1} +Z_m) /\varepsilon } +1) (\zeta_{m}e^{(-Z_m+mQ) /\varepsilon } -1)} -1 \Big\} , \nonumber 
%& \Big\{ \frac{e^{(A+2mQ )/\varepsilon} + \zeta_{m - 1} \zeta_{m}e^{(Z_{m-1} +Z_m) /\varepsilon } +1 - \zeta_{m - 1} e^{(Z_{m-1}+mQ) /\varepsilon } - \zeta_{m}e^{(-Z_m+mQ) /\varepsilon } }{(\zeta_{m - 1} \zeta_{m}e^{(Z_{m-1} +Z_m) /\varepsilon } +1) (\zeta_{m}e^{(-Z_m+mQ) /\varepsilon } -1)} \Big\} \nonumber \\
%& = \zeta_{m}  e^{-Z_{m} /\varepsilon } \Big\{ \frac{e^{(A+2mQ )/\varepsilon} }{(\zeta_{m - 1} \zeta_{m}e^{(Z_{m-1} +Z_m) /\varepsilon } +1) (\zeta_{m}e^{(-Z_m+mQ) /\varepsilon } -1)} -1 \Big\} \nonumber \\
%\Big\{ \frac{e^{(A+mQ )/\varepsilon} - (\zeta_{m - 1} \zeta_{m}e^{(Z_{m-1} +Z_m) /\varepsilon } +1) (\zeta_{m}e^{(-Z_m+mQ) /\varepsilon } -1)}{(\zeta_{m - 1} \zeta_{m}e^{(Z_{m-1} +Z_m) /\varepsilon } +1) (\zeta_{m}e^{(-Z_m+mQ) /\varepsilon } -1)} \Big\} , \nonumber 
\end{align}
and Eq(\ref{eq:Zm+1ue}) is interpretated by picking up dominant terms of the right hand side of Eq.(\ref{eq:qPIIsignamp}).

We have similar propositions for backward evolution.
Namely, Propositions \ref{prop:Zsing1} and \ref{prop:Zsing3} are true for the backward evolution by replacing $m-1$ (resp. $m+1$) by $m+1$ (resp. $m-1$).
Note that the master thesis of the first author \cite{Iga} describe the details of the case of backward evolution.

Under the assumption that a solution $(\zeta _m , Z_{m})$ of the single p-ud PII is given and there exists $m_0 \in \mathbb{Z}$ such that $Z_{m_0-1} +Z_{m_0} > 0 $, we construct the function $(\eta _m ,Y_{m})$ such that ($(\eta _m ,Y_{m})$, $(\zeta _m , Z_{m})$) is a solution of the simultaneous p-ud PII.
In order to satisfy $Z_{m_0-1} +Z_{m_0}=\max [Y_{m_0},0]$, we have $ Y_{m_0 }= Z_{m_0-1} +Z_{m_0} $ and it follows from Eq.(\ref{eq:evolp}) that $\eta _{m_0}= \zeta _{m_0 -1} \zeta _{m_0} $.
If $(\zeta _{m_0} ,Z_{m_0} )\neq (+1 , m_0 Q)$, then it follows from Eq.(\ref{eq:evolu2}) that $ Y_{m_0+1} =- (Z_{m_0-1} +Z_{m_0}) + A+ 2m_0 Q- \max [ m_0 Q-Z_{m_0}, 0 ] $, and if $(\eta _{m_0 +1} ,Y_{m_0 +1} )\neq (+1 , 0)$, then it follows from Eq.(\ref{eq:evolu1}) that 
\begin{equation}
Z_{m_0+1} =- Z_{m_0} + \max[ 0, A+ 2m_0 Q - (Z_{m_0-1} +Z_{m_0}) - \max [ m_0 Q-Z_{m_0}, 0 ]] .
\end{equation}
Hence the evolution coincides with Eq.(\ref{eq:Zm+1ue}) and we have $Z_{m_0} +Z_{m_0 +1} \geq 0 $.
We can also show that the sign $\zeta _{m_0 +1}$ coincides with the one in Proposition \ref{prop:singleudPIIevol}.
By repeating the argument, we obtain that if the evolution as the simultaneous equation is unique, then $(\zeta _{m_0+2}, Z_{m_0 +2})$ is written in the form of Proposition \ref{prop:singleudPIIevol} with $m=m_0+1$.
We also obtain that the function $(\zeta _{m+1} , Z_{m+1} )$ is also written in the form of Proposition \ref{prop:singleudPIIevol} as far as the forward evolution is unique, and it is also true for the backward evolution. 
Note that the condition $Z_{m_0-1} +Z_{m_0} \leq 0 $ causes the indefinite evolution for single p-ud PII. Namely we have 
\begin{prop} \label{prop:Zm0-1Zm0neg}
If $Z_{m_0-1} +Z_{m_0} \leq 0 $, then the indefinite forward evolution or the indefinite backward evolution for single p-ud PII occurs around $m=m_0$.
\end{prop}
\begin{proof}
Assume that $Z_{m_0+1}$ is determined uniquely.
It follows from $Z_{m_0-1} +Z_{m_0} \leq 0 $ that 
\begin{equation}
 Z_{m_0 +1} = -Z_{m_0} +\max [0, A+2m_0 Q -\max[0, m_0 Q-Z_{m_0} ] ] .
\end{equation}
The value $Z_{m_0 +1}$ is independent from the value $Z_{m_0 -1} (\leq -Z_{m_0})$.
On the backward evolution such that the values $Z_{m_0+1}$ and $Z_{m_0}$ are given, the value $Z_{m_0-1}$ is determined indefinitely.
\end{proof}
\begin{thm}
Assume that a solution of the simultaneous p-ud PII (Eqs.(\ref{eq:udYYZ}) and (\ref{eq:udYZZ})) is given and the function $(\zeta _{m} , Z_{m} )$ also satisfies the single p-ud PII (Eq.(\ref{eq:udPIIsingle})).
If the evolution by the simultaneous p-ud PII is indefinite at $m=m'$, then the evolution by the single p-ud PII is also indefinite around $m=m'$.
\end{thm}
\begin{proof}
If the evolution by the simultaneous equation is indefinite at $m=m'$, then it follows from Corollary \ref{cor:uniindef} that $(\zeta _{m''} , Z_{m''} ) = ( +1, m''Q)$ or $(\eta _{m''} , Y_{m''} ) = ( +1, 0) $ for $m'' =m'-1, m'$ or $m'+1$.
If $(\zeta _{m''} , Z_{m''} ) = ( +1, m''Q)$, then it follows from Proposition \ref{prop:Zsing3} that the evolution by Eq.(\ref{eq:udPIIsingle}) is also indefinite.
If $(\eta _{m''} , Y_{m''} ) = ( +1, 0) $, then it follows from Eq.(\ref{eq:udYZZ}) that
\begin{align}
\max[ 0 , Z_{m''}+Z_{m''-1}+S( -\zeta_{m''-1}\zeta_{m''}) ]  = \max[ 0 ,Z_{m''}+Z_{m''-1}+S(\zeta_{m''-1}\zeta_{m''}) ] ,
\end{align}
which implies $Z_{m''}+Z_{m''-1} \leq 0 $.
By Proposition \ref{prop:Zm0-1Zm0neg}, we have the theorem.
\end{proof}
On solutions of the single p-ud PII, we have 
\begin{thm} \label{thm:nonentire}
%Assume that $Q<0$.
There exists no solution to the single p-ud PII (Eq.(\ref{eq:udPIIsingle})) such that any forward and backward evolution for all $m \in \mathbb{Z}$ are unique. 
\end{thm}
\begin{proof}
Assume that there exists a solution to Eq.(\ref{eq:udPIIsingle}) such that any forward and backward evolution for all $m \in \mathbb{Z}$ are unique.
Then it follows from Proposition \ref{prop:Zm0-1Zm0neg} that $Z_m +Z_{m+1} >0$ for all $m \in \mathbb{Z}$.
Since $Z_{m +1} \neq -Z_m $ and $Z_{m-1} +Z_{m} >0$, we have 
\begin{equation}
Z_{m +1} = -Z_m + A+2mQ -\max[0, mQ-Z_m ] -(Z_{m-1} +Z_{m}) ,
\end{equation}
by Eq.(\ref{eq:Zm+1ue}).
Therefore $Z_{m +1} + 2Z_m +Z_{m-1} =A+2mQ $ or $Z_{m +1} + Z_m +Z_{m-1} =A+mQ $.
If we take the value $m$ sufficiently large, then $Z_{m +1} + 2Z_m +Z_{m-1} <0$ or $Z_{m +1} + Z_m +Z_{m-1} <0 $.
However it contradicts to $Z_m +Z_{m+1} >0$ for all $m \in \mathbb{Z}$.
\end{proof}

\section{Determinant-type solutions of simultaneous equations} \label{sec:etaYIIT}

The ultradiscrete limit of determinant-type solutions of $q$-PII with a parameter was obtained in \cite{IIT}.
We write it by setting $C= B-A-(m_0^2+m_0)Q$ and $\chi =\alpha \beta $ in \cite[Theorem 3]{IIT}.

Assume that $Q<0$ and the constant $A$ in the p-ud PII is written as $A=(2N+1)Q$ for $N \in {\mathbb Z}_{\geq 0}$.
Let $m_0$ be a negative integer satisfying 
\begin{align}
& m_0 \leq \min (-3N-2, -N(N+1)/2 -1) ,
\label{eq:m0N}
\end{align}
$k_0 \in \{ 0, 1 , \cdots , N \}$ and $C$ be values such that 
\begin{equation}
\begin{array}{ll}
-m_0 Q - (N-k_0 )(N-k_0 +1)Q  < C & \\
\qquad < -m_0 Q- (N-k_0 +1)(N-k_0 +2)Q (< (m_0+1)Q), &  (k_0 \neq 0), \\
-m_0 Q - N(N-1)Q < C < (m_0 +1) Q  & (k_0=0),
\end{array}
\label{eq:CineqN}
\end{equation}
and $\chi \in \{ +1, -1\}$.
Using these notation, the following function $(\zeta ^{(N)}(m) , Z ^{(N)}(m))$ was obtained in \cite[Theorem 3]{IIT} by the p-ud limit of a solution of $q$-PII in terms of determinants, and it satisfies the single p-ud PII (Eq.(\ref{eq:udPIIsingle})).\\
(I) If $m \leq m_0 - 2N -1$, then 
\begin{equation}
(\zeta^{(N)}(m) , Z^{(N)}(m))=(+1 , (-m-2N-1)Q).
\label{eq:IITNsolm<<} 
\end{equation}
(II) If $m_0 - 2N \leq m \leq m_0 + N - 3k_0 +1 $, then 
\begin{align}
& (\zeta^{(N)}(m) , Z^{(N)}(m))= \label{eq:IITsolZ1} \\
& \begin{cases}
((-1)^{j} \chi , -C-j^2 Q) & (m=m_0 -2N + 3j) \\
(+1 , (m_0+j+1)Q) & (m=m_0 -2N + 3j +1) \\
((-1)^{j} \chi , C+(j+1)^2 Q) & (m=m_0 -2N + 3j +2),  
\end{cases}
\nonumber 
\end{align}
where $0 \leq j \leq N- k_0 $ in the first and the second cases and $0 \leq j \leq N- k_0 -1 $ in the third case.\\
(III) If $m = m_0 + N - 3k_0 +2 $ and $k_0 \neq 0$, then 
\begin{equation}
(\zeta^{(N)}(m) , Z^{(N)}(m))
= (-1 , (-m_0-N+k_0 -1)Q) .
\end{equation}
(IV) If $m_0 + N - 3k_0 +3 \leq m \leq m_0 + N $ and $k_0 \neq 0$, then
\begin{align}
& (\zeta^{(N)}(m) , Z^{(N)}(m))
= \nonumber \\
&  \begin{cases}
(+1 , (m_0+j)Q) & (m=m_0 -2N +3j) \\
((-1)^{j} \chi , C +2 m_0 Q+(j+1)^2 Q ) & (m=m_0 -2N +3j+1)  \\
((-1)^{j+1} \chi , -C -2 m_0 Q-(j+1)^2 Q)  & (m=m_0 -2N +3j+2), 
\end{cases}
\end{align}
where $N -k_0 +1 \leq j \leq N $ in the first case and $N- k_0 +1 \leq j \leq N-1 $ in the second and the third cases.\\
(V) If $m_0 + N +1 \leq m \leq -2N-1$, then
\begin{align}
(\zeta^{(N)}(m) , Z^{(N)}(m))
=(+1 , mQ) .
\label{eq:IITNsolm>>} 
\end{align}
On the function $(\zeta^{(N)}(m) , Z^{(N)}(m))$, we can confirm the following properties:\\
(i) If $k_0 =0$, then $Z^{(N)}(m) < mQ$ for $ m \leq m_0 + N$ and $Z^{(N)}(m)= mQ$ for $m=m_0 +N+1$.\\ 
(ii) If $k_0 \neq 0$, then $Z^{(N)}(m) < mQ$ for $ m \leq m_0 + N -1 $ and $Z^{(N)}(m)= mQ$ for $m=m_0 +N$. 

We now calculate the function $(\eta ^{(N)}(m) , Y ^{(N)}(m))$ associated to the above solution $(\zeta ^{(N)}(m) , Z ^{(N)}(m))$.
\begin{prop} \label{prop::}
The function $(\eta ^{(N)}(m) , Y ^{(N)}(m))$ associated to the solution $(\zeta ^{(N)}(m) , Z ^{(N)}(m))$ in Eqs.(\ref{eq:IITNsolm<<})--(\ref{eq:IITNsolm>>}) is written as follows:\\
(i) If $m_0 - 2N \leq m \leq m_0 + N - 3k_0 +1 $, then 
\begin{align}
& (\eta^{(N)}(m) , Y^{(N)}(m))
= \label{eq:IITsolY1} \\
& \begin{cases}
(+1 , 2j Q ) & (m=m_0 -2N + 3j) \\
((-1)^{j} \chi  , -C+m_0 Q +(-j^2 +j+1)Q) & (m=m_0 -2N + 3j +1) \\
((-1)^{j} \chi  , C+m_0 Q + (j+1)(j+2) Q ) & (m=m_0 -2N + 3j +2)  ,
\end{cases}
\nonumber 
\end{align}
where $0 \leq j \leq N- k_0 $.\\
(ii) If $m_0 + N - 3k_0 +3 \leq m \leq m_0 + N $ and $k_0 \neq 0$, then
\begin{align}
& (\eta^{(N)}(m) , Y^{(N)}(m))
= \nonumber \\
&  \begin{cases}
((-1)^{j} \chi  , -C - m_0 Q - j(j -1) Q ) & (m=m_0 -2N +3j) \\
((-1)^{j} \chi  ,  C +3 m_0 Q+(j^2 +3j +1) Q ) & (m=m_0 -2N +3j+1)  \\
(+1 , 2(j+1) Q ) & (m=m_0 -2N +3j+2) ,
\end{cases}
\end{align}
where $N -k_0 +1 \leq j \leq N $ in the first case and $N- k_0 +1 \leq j \leq N-1 $ in the second and the third cases.
\end{prop}
\begin{proof}
We show (i).
We write $Z^{(N)}(m)$, $\zeta ^{(N)}(m)$, $Y^{(N)}(m)$ and $\eta ^{(N)}(m)$  as $Z_m$, $\zeta _m$, $Y_m$ and $\eta _m$ respectively.
Recall that $Z_{m+1} + Z_m = \max[Y_{m+1}, 0] $ and $\zeta _{m+1} \zeta _m =\eta _{m+1}$ for $Y_{m+1} >0$.
Since $Z_{m_0 -2N + 3j} + Z_{m_0 -2N + 3j +1} = -C-j^2 Q + (m_0+j+1)Q >0$ and $Z_{m_0 -2N + 3j+1} + Z_{m_0 -2N + 3j +2} = C+(j+1)^2 Q + (m_0+j+1)Q>0$, we have
\begin{align}
& \eta_{m_0 -2N + 3j +1} = (-1)^j \chi, \;  Y_{m_0 -2N + 3j +1} = -C-j^2 Q + (m_0+j+1)Q, \\
& \eta_{m_0 -2N + 3j +2} = (-1)^j \chi, \; Y_{m_0 -2N + 3j +2} = C+(j+1)^2 Q + (m_0+j+1)Q. \nonumber
\end{align}
It follows from $Z_m < m Q $ that $Y_{m_0 -2N + 3j+1} + Y_{m_0 -2N + 3j} = (2N+1)Q + (m_0 -2N + 3j ) Q +Z_{m_0 -2N + 3j}$ and we have $Y_{m_0 -2N + 3j} = 2j Q  $.
We also have $\eta _{m_0 -2N + 3j} =+1$.

(ii) is shown similarly.
\end{proof}
On the function $Y^{(N)}(m )$, we have $Y^{(N)}(m_0 -2N ) =0$ and 
%, $Y^{(N)}(m)<0$ for $m=m_0 -2N + 3j$ ($j=1,2, \dots , N -k_0$) and $Y^{(N)}(m)>0$ for $m= m_0 -2N + 3j +1 $ $(j=0,1, \dots , N -k_0 )$ or $m= m_0 -2N + 3j +2 $ $(j=0,1, \dots , N -k_0 -1)$.
\begin{equation}
\begin{array}{ll}
Y^{(N)}(m) < 0 & (m= m_0 -2N + 3j, \; j=1,2, \dots , N -k_0), \\
Y^{(N)}(m ) > 0 & (m= m_0 -2N + 3j +1 , \; j=0,1, \dots , N -k_0), \\
Y^{(N)}(m ) > 0 & (m= m_0 -2N + 3j +2 , \; j=0,1, \dots , N- k_0 -1). 
\end{array}
\end{equation}
%\begin{equation}
%\begin{array}{ll}
%Y^{(N)}(m_0 -2N + 3j) < 0 & (j=1,2, \dots , N -k_0), \\
%Y^{(N)}(m_0 -2N + 3j +1 ) > 0 & (j=0,1, \dots , N -k_0), \\
%Y^{(N)}(m_0 -2N + 3j +2 ) > 0 & (j=0,1, \dots , N- k_0 -1). 
%\end{array}
%\end{equation}
If $k_0 \neq 0$, then $Y^{(N)}(m_0 +N-3k_0 +2)<0 $, $Y^{(N)}(m_0 +N-3k_0 +3)<0 $ and
\begin{equation}
\begin{array}{ll}
Y^{(N)}(m) > 0 & (m= m_0 -2N + 3j ,\; j= N -k_0 +2, \dots , N), \\
Y^{(N)}(m) > 0 & (m= m_0 -2N + 3j+1 ,\; j=N -k_0 +1, \dots , N-1),  \\
Y^{(N)}(m ) < 0 & (m= m_0 -2N + 3j+2 ,\; j=N -k_0 +1, \dots , N-1). 
\end{array}
\end{equation}
By applying Corollary \ref{cor:uniindef}, uniqueness and indefiniteness of the solution of simultaneous p-ud PII can be described.
\begin{prop}
(i) If $k_0 \neq 0$, then we have unique evolution on $Y^{(N)}(m_0 -2N) , Z^{(N)}(m_0 -2N) , Y^{(N)}(m_0 -2N +1) , \dots ,Y^{(N)}(m_0 +N)  ,  Z^{(N)}(m_0+N) $ and indefinite evolution occurs on determining $Z^{(N)}(m_0 -2N-1)$ and $Y^{(N)}(m_0 +N+1)$.\\
(ii) If $k_0 = 0$, then we have unique evolution on $Y^{(N)}(m_0 -2N) , Z^{(N)}(m_0 -2N) , Y^{(N)}(m_0 -2N +1) , \dots ,Y^{(N)}(m_0 +N+1)  ,  Z^{(N)}(m_0 +N+1)$ and indefinite evolution occurs on determining $Z^{(N)}(m_0 -2N-1)$ and $Y^{(N)}(m_0 +N+2)$.
\end{prop}
An example of the ultradiscrete limit of a determinant-type solution $(\zeta^{(N)}(m), Z^{(N)}(m) )$ for $N=3$ and $Q=-3$ was given in Eq.(\ref{eq:z3}) in the introduction and the function $(\eta ^{(N)}(m) , Y ^{(N)}(m))$ associated to the solution was given in Eq.(\ref{eq:y3}).

In \cite{Iga}, the ultradiscrete limit of determinant-type solutions of $q$-PII for the case $Q<0$ and $A=(2M+1)Q$, $M \in {\mathbb Z}_{\leq -1}$ was obtained, and we review it in the appendix.
We describe the ultradiscrete function here.
Let $m_0$ be a value satisfying 
\begin{align}
m_0 \leq \min ( 3M+1, -M(M+1)/2 -1 ) , \label{eq:m0M}
\end{align}
$k_0 \in \{ 0, 1 , \cdots , -M-1 \}$ and $C$ be values such that 
\begin{equation}
\begin{array}{ll}
-m_0 Q - (M+k_0+1)(M+k_0)Q  \leq C & \\
\qquad < -m_0 Q - (M+k_0)(M+k_0-1)Q , &  k_0 \neq 0, \\
-m_0 Q- M(M+1)Q < C < (m_0+1)Q , & k_0 = 0,
\end{array}
\label{eq:CineqM}
\end{equation}
and $\chi \in \{ +1, -1\}$.
Using these notation, the following function $(\zeta ^{(M)}(m) , Z ^{(M)}(m))$ is obtained by setting $C= B-A-(m_0^2+m_0)Q$ and $\chi =\alpha \beta $ in Theorem \ref{thm:zpum_0_M} in the appendix, and it satisfies p-ud PII (Eq.(\ref{eq:udPIIsingle})).
\begin{prop} \label{prop:zpum_0_M}
(i) If $m \leq m_0 + M$, then
\begin{equation}
(\zeta^{(M)}(m), Z^{(M)}(m)) = (+1, mQ) . \label{eq:IgaMsolm<<}
\end{equation}
(ii) If $m_0 + M + 1 \leq m \leq m_0-2M-3k_0$ $(k_0\neq 0)$ or $m_0 + M + 1 \leq m \leq m_0-2M-1$ $(k_0=0)$, then
\begin{align}
& (\zeta^{(M)}(m), Z^{(M)}(m)) = \nonumber \\
& \begin{cases}
(+1, (m_0+M+j)Q) & (m=m_0 + M + 3j - 2) \\
((-1)^{j-1} \chi , C+(j^2+M)Q) & (m=m_0 + M + 3j -1) \\
((-1)^{j-1} \chi , -C+(-j^2+2j+M)Q) & (m=m_0 + M + 3j),
\end{cases}
\end{align}
$(1 \leq j \leq -M-k_0)$.
The case $k_0=0$ and $m=m_0-2M$ is included in (iv).\\
(iii) If $m_0 - 2M - 3k_0 + 1 \leq m \leq m_0 - 2M - 2$ and $k_0 \neq 0$, then
\begin{align}
& (\zeta^{(M)}(m), Z^{(M)}(m)) = \nonumber \\
&  \begin{cases}
((-1)^{j-1} \chi , C+(2m_0+j^2+M)Q) & (m=m_0+M+3j-2) \\
((-1)^{j-1} \chi , -C+(-2m_0-j^2+2j+M)Q) & (m=m_0+M+3j-1)  \\
(+1, (m_0+M+j)Q) & (m=m_0+M+3j).
%$C= B-A-(m_0^2+m_0)Q$
\end{cases}
\end{align}
Here $-M-k_0+1 \leq j \leq -M$ for the first case and $-M-k_0+1 \leq j \leq -M-1$ for the second and the third cases.\\
(iv) If $m_0 - 2M -1 \leq m \leq M+1$ $(k_0\neq 0)$ or $m_0 - 2M \leq m \leq M+1$ $(k_0=0)$, then
\begin{align}
(\zeta^{(M)}(m), Z^{(M)}(m))
=(+1, (-m-2M-1)Q).
\label{eq:IgaMsolm<}
\end{align}
\end{prop}
The function $Z^{(M)}(m)$ in Proposition \ref{prop:zpum_0_M} satisfies $Z^{(M)}(m) =mQ $ for $m \leq m_0 + M + 1$ and
\begin{equation}
\begin{array}{ll}
Z^{(M)}(m) >mQ & (m=m_0 + M + 3j - 2, \; j= 2,3, \dots , -M-k_0 ),\\
Z^{(M)}(m) <mQ & (m=m_0 + M + 3j - 1, \; j= 1,2, \dots ,-M-k_0 ), \\
Z^{(M)}(m) <mQ & (m=m_0 + M + 3j , \; j= 1,2, \dots , -M-k_0-1) .
\end{array}
\end{equation}
If $k_0 \neq 0$, then $Z^{(M)}(m) >mQ $ for $m=m_0 -2M - 3k_0 , \: m_0 -2M - 3k_0 +1$, and
\begin{equation}
\begin{array}{ll}
Z^{(M)}(m) < mQ & (m=m_0 + M + 3j - 2, \; j=  -M-k_0+2, \dots ,-M ),\\
Z^{(M)}(m) < mQ & (m=m_0 + M + 3j - 1, \; j=  -M-k_0+1, \dots ,-M ), \\
Z^{(M)}(m) > mQ & (m=m_0 + M + 3j , \; j=  -M-k_0+1, \dots ,-M ).
\end{array}
\end{equation}
We now calculate the function $(\eta ^{(M)}(m) , Y ^{(M)}(m))$ associated to the above solution $(\zeta ^{(M)}(m) , Z ^{(M)}(m))$.
\begin{prop} \label{prop:::}
The function $(\eta ^{(M)}(m) , Y ^{(M)}(m))$ associated to the function $(\zeta ^{(M)}(m) , Z ^{(M)}(m))$ in Proposition \ref{prop:zpum_0_M} is written as follows:\\
(i) If $m_0 + M + 2 \leq m \leq m_0-2M-3k_0$, then
\begin{align}
& (\eta^{(M)}(m), Y^{(M)}(m)) = \nonumber \\
& \begin{cases}
((-1)^{j-1} \chi , C+ m_0 Q + 2M Q + j(j+1)Q) & (m=m_0 + M + 3j -1) \\
(+1 , 2(M+j)Q) & (m=m_0 + M + 3j) \\
((-1)^{j-1} \chi , -C + m_0 Q +2MQ +(-j^2+3j+1)Q) & (m=m_0 + M + 3j +1 ), \\
\end{cases}
\end{align}
where $1 \leq j \leq -M- k_0 $ in the first and the second cases and $0 \leq j \leq -M- k_0 -1 $ in the third case.\\
(ii) If $k_0 \neq 0$, then 
\begin{align}
 (\eta^{(M)}(m), Y^{(M)}(m)) = & (-1 ,(2m_0 - 2M -4 k_0 +1) Q ) \\
& \qquad \qquad \qquad (m= m_0 -2M -3k_0 +1). \nonumber
\end{align}
(iii) If $m_0 - 2M - 3k_0 + 2 \leq m \leq m_0 - 2M - 1$ and $k_0 \neq 0$, then
\begin{align}
& (\eta^{(M)}(m), Y^{(M)}(m)) = \nonumber \\
&  \begin{cases}
(+1 , 2(M+j)Q ) & (m=m_0+M+3j-1)  \\
((-1)^{j-1} \chi ,  -C+(-m_0+2M-j^2+3j)Q ) & (m=m_0+M+3j) \\
((-1)^{j-1} \chi , C+(3m_0+2M+ j^2+3j+1)Q ) & (m=m_0+M+3j+1) ,
\end{cases}
\end{align}
where $-M-k_0+1 \leq j \leq -M$ for the first case and the second cases, and $-M-k_0+1 \leq j \leq -M-1$ for the third case.
\end{prop}
\begin{proof}
(i) and (iii) are shown by using $Z_{m+1} + Z_m >0 $, $Z_{m+1} + Z_m = \max[Y_{m+1}, 0] $ and $\zeta _{m+1} \zeta _m =\eta _{m+1}$ for $Y_{m+1} >0$. 
We show (ii).
In the case $j=-M-k_0$, we have $Z_{m_0 +M +3j } -(m_0+M+3j)Q = -C -m_0 Q +(M+k_0 ) (M+k_0-1)Q >0$.
By applying the evolution of p-ud PII, we have $Y_{m_0 -2M-3k_0+1} = -Y_{m_0 +M +3(-M-k_0) } + (2M+1)Q +2(m_0 -2M-3k_0+1)Q $ and $\eta _{m_0 -2M -3k_0 +1} =-\eta _{m_0 -2M -3k_0 }$. Therefore we have (ii).
\end{proof}
%If $k_0 =0$, then $Y^{(M)}(m)>0$ for $m_0 + M + 1 \leq m \leq m_0-2M-1$ and $Y^{(M)}(m)=0$ for $m=m_0 -2M$. 
If $k_0 \neq 0$ (resp. $k_0 =0$), then $Y^{(M)}(m)>0$ for $m_0 + M + 1 \leq m \leq m_0-2M-2$ (resp. $m_0 + M + 1 \leq m \leq m_0-2M-1$) and $Y^{(M)}(m)=0$ for $m=m_0 -2M-1$ (resp. $m=m_0 -2M$).
\begin{prop}
(i) If $k_0 \neq 0$, then we have unique evolution on $Z^{(M)}(m_0 +M+1) , Y^{(M)}(m_0 +M+2) , Z^{(M)}(m_0 +M+2) , \dots ,Z^{(M)}(m_0 -2M - 2)  ,  Y^{(M)}(m_0 -2M - 1) $ and indefinite evolution occurs on determining $Y^{(M)}(m_0 +M+1 )$ and $Z^{(M)}(m_0 -2M-1 )$.
\\
(ii) If $k_0 = 0$, then we have unique evolution on $Z^{(M)}(m_0 +M+1) , Y^{(M)}(m_0 +M+2) , Z^{(M)}(m_0 +M+2) , \dots ,Z^{(M)}(m_0 -2M - 1)  ,  Y^{(M)}(m_0 -2M ) $ and indefinite evolution occurs on determining $Y^{(M)}(m_0 +M+1)$ and $Z^{(M)}(m_0 -2M)$.
\end{prop}

We show an example.
If $M=-3$, $Q=-2$, $m_0=-10$, $k_0=1$, $C=-10$, $\chi =+1$, then the function $(\zeta^{(M)}(m), Z^{(M)}(m)) $ in Proposition \ref{prop:zpum_0_M} is written as
\begin{align}
& (\zeta^{(-3)}(m), Z^{(-3)}(m))= \left\{
\begin{array}{ll}
 (+1,-2m) & (m\leq -13) \\
 (+1,24) & (m=-12) \\
 (+1,-6) & (m=-11) \\
 (+1,14) & (m=-10) \\
 (+1, 22) & (m=-9 ) \\
 (-1, -12) & (m=-8 ) \\
 (-1, 16) & (m=-7 ) \\
 (+1, 18) & (m=-6 ) \\
 (+1, 2m-10 ) & (-5 \leq m \leq -2 ) .
\end{array}
\right. \label{eq:ZM-3}
\end{align}
Note that we fixed an error in \cite{Iga}.
It follows from Proposition \ref{prop:::} that 
\begin{align}
( \eta^{(-3)}(m), Y^{(-3)}(m))= \left\{
\begin{array}{ll}
(+1,18) & (m=-11) \\
(+1,8) & (m=-10) \\
(+1,36) & (m=-9 ) \\
(-1,10) & (m=-8 ) \\
(+1,4) & (m=-7 ) \\
(-1,34) & (m=-6 ) \\
(+1,0) & (m=-5 ) .
\end{array}
\right.\label{eq:YM-3}
\end{align}

\section{Two parameter solutions} \label{sec:twoparameter}
In this section, we investigate two parameter solutions of the simultaneous p-ud PII (Eqs.(\ref{eq:udYYZ}) and (\ref{eq:udYZZ})).
Firstly we investigate solutions under the condition 
\begin{equation}
Y_m< 0, \; Z_m > mQ \quad (m=m', m'+1,\dots ).
\end{equation}
Then it follows from Eqs.(\ref{eq:evolu2})--(\ref{eq:evolp}) that the simultaneous p-ud PII is written as
\begin{align}
& Y_{m+1} + Y_m = A + 2m Q , \; Z_{m+1} + Z_m = 0, \label{eq:sol>>eqs} \\
& \eta _{m+1} \eta _m =-1, \; \zeta _{m+1} \zeta _m =-1 . \nonumber
\end{align}
By setting $Z_{m'}=d_1$, $Y_{m'} = ((2m'-1)Q+A)/2 +d_2$, $\zeta _{m'} =\zeta $ and $\eta _{m'} =\eta $, it is solved as
\begin{align}
& (\eta _m , Y_{m} ) = ((-1)^{m-m'} \eta , \frac{(2m-1)Q+A}{2} +d_2 (-1)^{m-m'}), \label{eq:sol01} \\
& ( \zeta _m, Z_m )=((-1)^{m-m'} \zeta , d_1 (-1)^{m-m'}), \nonumber
\end{align}
and satisfies $Y_{m+2} = Y_{m} +2Q$ and $Z_{m+2}= Z_{m}$.
Namely it has 2-periodic linear structure.
By combining with the condition $Y_m< 0$ and $Z_m > mQ$, we have the following proposition:
\begin{prop} \label{prop:sol>>} (Type $++$)\\
(i) Let $d_1$ and $d_2$ be real constants and $\eta , \zeta \in \{ \pm 1 \}$.
If $Q<0$ and the integer $m'$ satisfies
\begin{align}
& m'Q < \min (d_1,-d_1 - Q) , \label{eq:sol>>cond} \\
& 2m'Q < -Q -A+\min (d_2, -d_2-2Q) , \nonumber
\end{align}
then the functions $( \eta _m, Y_m )$ and $( \zeta _m, Z_m )$ defined by Eq.(\ref{eq:sol01}) satisfy p-ud PII (Eqs.(\ref{eq:udYYZ}) and (\ref{eq:udYZZ})) for $m \geq m'$.\\
(ii) If $Q<0$ and there exists a integer $m'$ such that the functions $( \eta _m, Y_m )$ and $( \zeta _m, Z_m )$ satisfy p-ud PII for $m \geq m' $ and $Y_m< 0$ and $Z_m > mQ$ for $m=m'$ and $m'+1$, then they are written in the form of Eq.(\ref{eq:sol01}) and satisfy $Y_m< 0$ and $Z_m > mQ$ for $m \geq m'$.
\end{prop}
\begin{proof}
(i) If $m \geq m'$, then it follows from Eq.(\ref{eq:sol>>cond}) that the functions $Y_m$ and $Z_m$ in Eq.(\ref{eq:sol01}) satisfy $Y_m< 0$ and $Z_m > mQ$. We can confirm Eq.(\ref{eq:sol>>eqs}) directly.

(ii) It follows from the assumption that $Y_{m'+2} = Y_{m'} +2Q $ and $Z_{m'+2}= Z_{m'}$. Hence we have $ Y_{m'+2} <2Q <0$, $Z_{m'+2} >m'Q >(m'+2)Q$, Eq.(\ref{eq:sol>>eqs}) for $m=m'+2$ and $ Y_{m'+3}= Y_{m' +1} +2Q  <0$.
Similarly we have $Z_{m'+3} >(m'+3)Q$, Eq.(\ref{eq:sol>>eqs}) for $m=m'+3$ and $ Y_{m'+4}= Y_{m' +2} +2Q  <0$.
Thus we have (ii) inductively.
\end{proof}
We denote the solution in Proposition \ref{prop:sol>>} by Type ++.
Note that the solution in the form of Eq.(\ref{eq:sol01}) was essentially obtained by Murata \cite{Mu} for the case of ultradiscrete PII without parity variables.

Next we investigate solutions of p-ud PII under the condition 
\begin{equation}
Y_m>0, \;  Z_m <mQ \quad (m=m', m'-1,\dots ).
\end{equation}
Then we have 
\begin{align}
& Z_{m-1} + Z_m = Y_{m}, \; Y_{m-1} + Y_m = Z_{m-1} + A + (m-1) Q, \\
& \zeta _{m-1} \zeta _m =\eta _{m}, \; \eta _{m-1} \eta _m =\zeta _{m-1} . \nonumber
\end{align}
By setting $Z_{m'}= C+ (m'Q+A)/3 $, $Y_{m'}= -D+((2m'-1)Q+2A)/3 $, $\zeta _{m'} =\zeta $ and $\eta _{m'} =\eta $, we have 
\begin{align}
& Z_{m' -1} = Y_{m'} - Z_{m'} =  \frac{(m'-1)Q+A}{3} - C- D ,\\
& Y_{m'-1} = -Y_{m'} + Z_{m' -1} + A + (m' -1) Q = \frac{(2m'-3)Q +2A}{3} - C, \nonumber \\
& Z_{m' -2} = \frac{(m'-2)Q+A}{3} + D , \; Y_{m'-2} = \frac{(2m'-5)Q +2A}{3} + C +D ,\nonumber \\
& Z_{m' -3} = \frac{(m'-3)Q+A}{3} + C ,\; Y_{m'-3} =  \frac{(2m'-7)Q +2A}{3} -D . \nonumber 
\end{align}
Thus
\begin{equation}
Y_{m'-3} = Y_{m'} -2Q , \; Z_{m'-3}= Z_{m'} -Q ,
\end{equation}
i.e.~the solution has the 3-periodic linear structure.
By combining with the condition $Y_m > 0$ and $Z_m < mQ$ and discussion on parity variables, we have the following proposition:
\begin{prop} \label{prop:sol<<} (Type $--$)\\
(i) Let $c_m$ be a real sequence such that $c_{m-2}+c_{m-1}+c_{m}=0$ and $\eta , \zeta \in \{ \pm 1 \}$.
If $Q<0$ and the integer $m'$ satisfies
\begin{align}
& 2m'Q > Q-2A+\max (3c_{m'+1},3c_{m'}+2Q,3c_{m'-1}+4Q) , \\
& 2m'Q > 2Q +A+\max (3c_{m'-1},3c_{m'+1}+2Q,3c_{m'}+4Q) , \nonumber
\end{align}
then the functions $( \zeta _m, Z_m )$ and $( \eta _m, Y_m )$ defined by
\begin{align}
& Z_m= \frac{mQ+A}{3} +c_{m} ,\quad Y_{m}= \frac{(2m-1)Q+2A}{3} - c_{m+1}, \label{eq:sol02} \\
& \zeta _{m'-3k} =\zeta , \; \eta _{m'-3k} =\eta , \; \zeta _{m'-3k-1} = \zeta \eta , \nonumber \\ 
& \eta _{m'-3k-1} =\zeta , \; \zeta _{m'-3k-2} =\eta , \; \eta _{m'-3k-2} = \zeta \eta , \; (k \in {\mathbb Z}_{\geq 0} ) \nonumber
\end{align}
satisfy p-ud PII (Eqs.(\ref{eq:udYYZ}) and (\ref{eq:udYZZ})) for $m\leq m' -1$.\\
(ii) If $Q<0$ and there exists a integer $m'$ such that the functions $( \zeta _m, Z_m )$ and $( \eta _m, Y_m )$ satisfy p-ud PII  for $m \leq m' -1$ and $Y_m > 0$ and $Z_m < mQ$ for $m=m', m'-1$ and $m'-2$, then the solution is written as Eq.(\ref{eq:sol02}) and we have $Y_m > 0$ and $Z_m < mQ$ for $m \leq m'$.
\end{prop}
\begin{proof}
It is proved similarly to Proposition \ref{prop:sol>>} by applying the backward evolution.
\end{proof}
Note that the solution in the form of Eq.(\ref{eq:sol02}) was essentially obtained by Murata \cite{Mu}, and the condition $c_{m-2}+c_{m-1}+c_{m}=0$ for $m  \leq m'$ implies that $c_{m-3} = c_{m}$ for $m  \leq m'$.

We are going to find other patterns of solutions by modifying Proposition \ref{prop:sol<<}.
%We are going to find solutions which are close to the one in Proposition \ref{prop:sol<<}.
If the condition $Y_m > 0$ and $Z_m < mQ$ for $m=m', m'-1$ and $ m'-2$ is satisfied, then the solution of p-ud PII for $m \leq m'-1$ is determined as Proposition \ref{prop:sol<<}.
Therefore we impose the condition as Eq.(\ref{eq:condZ---Y-++}) or Eq.(\ref{eq:condZ+--Y+++}).
%We investigate two cases (Eqs.(\ref{eq:condZ---Y-++}) and (\ref{eq:condZ+--Y+++})) which are close to the case in Proposition \ref{prop:sol<<}.
Let $K \in \mathbb{Z}_{\geq 0}$.
In the case 
\begin{align}
& Z_m<mQ, \; (m' \leq m \leq m'+3K+2), \label{eq:condZ---Y-++} \\
& Y_{m' + 3 k + 1} >0 , \; Y_{m' + 3 k + 2} >0 , \; Y _{m' + 3 k + 3} <0, \; (k=0,1,\dots ,K), \nonumber
\end{align}
we have 
\begin{align}
& Y_{m+1} + Y_m = Z_m + A + m Q , \; (m' \leq m \leq m'+3K+2), \\
& Z_{m+1} + Z_m = Y_{m+1}, \; (m=m' + 3 k + 1, \: m' + 3 k + 2, \: k=0,1,\dots ,K), \nonumber \\
& Z_{m+1} + Z_m = 0 , \; (m=m' + 3 k + 3, \: k=0,1,\dots ,K) \nonumber 
\end{align}
and the corresponding equations for the parity variables.
By considering the forward evolution from $(\eta _{m' } , Y_{m' } )= (\eta , D')  $ and $(\zeta _{m' },Z_{m' } )= (\zeta , - C' ) $, we have the following proposition;
\begin{prop} \label{prop:sol-A} (Type $-A$) \\
Let $\eta ,\zeta \in \{ \pm 1 \}$ and $C'$ and $D'$ be the constants satisfying the condition in Eq.(\ref{eq:condZ---Y-++}).
Then the functions $(\eta _{m } , Y_{m } )$ $(m'\leq m \leq m'+3K+3)$ and $(\zeta _{m },Z_{m } )$ $(m'\leq m \leq m'+3K+3)$ defined by 
\begin{align}
& (\eta _{m' + 3 k } , Y_{m' + 3 k } )= (\eta , 2 k Q + D'), \label{eq:solZ---Y-++} \\
& (\zeta _{m' + 3 k },Z_{m' + 3 k } )= (\zeta (-\eta )^{k} , - k ^2 Q - C' - k D' ), \nonumber \\
& (\eta _{m' + 3 k +1} ,Y_{m' + 3 k +1}) = (\eta \zeta (-\eta )^k , (m' - k ^2 +  k ) Q +A  - C' -(k+1)D'),  \nonumber \\
& (\zeta _{m' + 3 k +1} , Z_{m' + 3 k +1}) = (\eta , (m' +  k ) Q +A  -D'),  \nonumber \\
& (\eta _{m' + 3 k +2}  ,Y_{m' + 3 k +2} )=  (\zeta (-\eta )^k , (m' +k^2+ 3 k +1) Q  + A +C' +k D') , \nonumber \\
& (\zeta _{m' + 3 k +2} , Z_{m' + 3 k +2} )= (\eta \zeta (-\eta )^k , (k+1)^2  Q   +C' +(k+1) D') \nonumber
\end{align}
%for $k=0,1,\dots ,K$ and $k'=0,1, \dots ,K+1$ 
satisfy p-ud PII.
\end{prop}
Note that the solution in Eq.(\ref{eq:solZ---Y-++}) has partly quadratic terms in the independent variable as well as the solutions in Eqs.(\ref{eq:solZ+--Y+++}), (\ref{eq:solZ-+Y--}) and (\ref{eq:solZ++Y+-}).

Next we investigate the case 
\begin{align}
& Y_m >0, \; (m' +1 \leq m \leq m'+3K+3) \label{eq:condZ+--Y+++} \\
& Z_{m' + 3 k + 1} < (m' + 3 k + 1) Q, \; Z_{m' + 3 k + 2} <(m' + 3 k + 2) Q, \nonumber \\
& Z_{m' + 3 k + 3} >(m' + 3 k + 3) Q ,  \; (k=0,1,\dots ,K). \nonumber
\end{align}
\begin{prop} \label{prop:sol-B} (Type $-B$) \\
Assume that the constants $C'$ and $D'$ satisfy the condition in Eq.(\ref{eq:condZ+--Y+++}).
Then the functions $(\eta _{m } , Y_{m } )$ $(m'+1 \leq m \leq m'+3K+4)$ and $(\zeta _{m },Z_{m } )$ $(m'\leq m \leq m'+3K+3)$ defined by 
%The function
\begin{align}
& (\zeta _{m' + 3 k },Z_{m' + 3 k } )= (\zeta , (m'+k) Q+D' ), \label{eq:solZ+--Y+++} \\
& (\eta _{m' + 3 k +1} ,Y_{m' + 3 k +1}) = (-\eta (-\zeta )^{k}  , (m'+k(k +3)) Q-k D' +C' ),  \nonumber \\
& (\zeta _{m' + 3 k +1} , Z_{m' + 3 k +1}) = (-\eta \zeta (-\zeta )^k  , k (k+2) Q -(k +1) D' +C' ),  \nonumber \\
& (\eta _{m' + 3 k +2}  ,Y_{m' + 3 k +2} )=  (\zeta  , (2 k +1) Q+ A -D') , \nonumber \\
& (\zeta _{m' + 3 k +2} , Z_{m' + 3 k +2} )= (-\eta (-\zeta )^k  , (-k^2 +1) Q+ A+k D' -C' ) , \nonumber \\
& (\eta _{m' + 3 k +3} ,Y_{m' + 3 k +3} )= (\eta (-\zeta )^{k+1} , (m' -(k+1)(k-2)) Q+ A+ (k+1) D'-C' ) \nonumber
\end{align}
%for $k=0,1,\dots ,K$ and $k'=0,1, \dots ,K+1$ 
satisfy p-ud PII.
% if the constants $C'$ and $D'$ satisfy the condition in Eq.(\ref{eq:condZ+--Y+++}).
\end{prop}

We are going to find other patterns of solutions by modifying Proposition \ref{prop:sol>>}.
%We are going to find patterns of solutions which are close to the one in Proposition \ref{prop:sol>>}.
If the condition $Y_m < 0$ and $Z_m > mQ$ for $m=m'$ and $m'+1$ is satisfied, then the solution of p-ud PII for $m\geq m' $ is determined as Proposition \ref{prop:sol>>}.
Therefore we impose the condition as Eq.(\ref{eq:condZ-+Y--}) or Eq.(\ref{eq:condZ++Y+-}).
%We investigate two cases (Eqs.(\ref{eq:condZ-+Y--}) and (\ref{eq:condZ++Y+-})) which are close to the case in Proposition \ref{prop:sol>>}.
First we investigate the case 
\begin{align}
& Z_{m' + 2 k } < (m' + 2 k )Q , \; Z_{m' + 2 k +1} > (m' + 2 k +1)Q , \; (k=0,1,\dots ,K), \label{eq:condZ-+Y--} \\ 
& Y_m <0, \; (m' +1 \leq m \leq m'+2K+2) . \nonumber 
\end{align}
\begin{prop} \label{prop:sol+A} (Type $+A$) \\
Assume that the constants $C'$ and $D'$ satisfy the condition in Eq.(\ref{eq:condZ-+Y--}).
Then the functions $(\eta _{m } , Y_{m } )$ $(m' \leq m \leq m'+2K+2)$ and $(\zeta _{m },Z_{m } )$ $(m'\leq m \leq m'+2K+2)$ defined by 
%The function
\begin{align}
&  (\eta _{m' + 2 k } ,Y_{m'+2 k } )=( \eta (- \zeta )^{k }, k (m' +k +1) Q+k D'+C' ) , \label{eq:solZ-+Y--} \\
&  (\zeta _{m' + 2 k } ,Z_{m'+2 k } )=(\zeta , -D' ) , \nonumber  \\
&  (\eta _{m' + 2 k+1 } ,Y_{m'+2 k+1} )=( \zeta \eta (- \zeta )^k, -(k-1)(m'+k) Q +A - (k+1) D' -C' )  ,\nonumber \\
&  (\zeta _{m' + 2 k+1 } ,Z_{m'+2 k+1} )=( - \zeta ,D' ) \nonumber 
\end{align}
satisfy p-ud PII.
%for $k=0,1,\dots ,K$ and $k'=0,1, \dots ,K+1$ satisfies p-ud PII, if the constants $C'$ and $D'$ satisfy the condition in Eq.(\ref{eq:condZ-+Y--}
\end{prop}

Next we investigate the case 
\begin{align}
& Z_m > mQ, \; (m' \leq m \leq m'+2K+1), \label{eq:condZ++Y+-} \\
&  Y_{m' + 2 k + 1} <0 , \; Y _{m' + 2 k + 2} >0 , \; (k=0,1,\dots ,K). \nonumber
\end{align}
\begin{prop} \label{prop:sol+B} (Type $+B$) \\
Assume that the constants $C'$ and $D'$ satisfy the condition in Eq.(\ref{eq:condZ++Y+-}).
Then the functions $(\eta _{m } , Y_{m } )$ $(m' \leq m \leq m'+2K+2)$ and $(\zeta _{m },Z_{m } )$ $(m'\leq m \leq m'+2K+2)$ defined by 
%The function 
\begin{align}
&  (\eta _{m' + 2 k } ,Y_{m'+2 k } )=( \eta ,2 k Q + D'), \label{eq:solZ++Y+-} \\
&  (\zeta _{m' + 2 k } ,Z_{m'+2 k } )=(\zeta (-\eta )^{k} , -(k-1)k Q + C' - k D' ) , \nonumber  \\
&  (\eta _{m' + 2 k+1 } ,Y_{m'+2 k+1} )=( -\eta , 2 (m' + k ) Q +A  - D' )  ,\nonumber \\
&  (\zeta _{m' + 2 k+1 } ,Z_{m'+2 k+1} )=( -\zeta (-\eta )^k ,k (k +1) Q -C' +(k+1) D' )  \nonumber
\end{align}
satisfy p-ud PII.
%for $k=0,1,\dots ,K$ and $k'=0,1, \dots ,K+1$  satisfies p-ud PII, if the constants $C'$ and $D'$ satisfy the condition in Eq.(\ref{eq:condZ++Y+-}).
\end{prop}

\section{Perturbed solutions} \label{sec:perturb}

We investigate solutions of p-ud PII which are obtained by perturbing the initial value to the function ($(\eta^{(N)}(m) , Y^{(N)}(m))$, $(\zeta^{(N)}(m) , Z^{(N)}(m))$) or ($(\eta^{(M)}(m) , Y^{(M)}(m))$, $(\zeta^{(M)}(m) , Z^{(M)}(m))$) in section \ref{sec:etaYIIT}.

Let us recall the situation that the function $(\zeta^{(N)}(m) , Z^{(N)}(m))$ in Eqs.(\ref{eq:IITNsolm<<})--(\ref{eq:IITNsolm>>}) was described.
Assume that the constant $A$ in p-ud PII is written as $A=(2N+1)Q$ ($N \in {\mathbb Z}_{\geq 1}$), $k_0 \in \{ 0, 1 , \cdots , N \}$, $\chi \in \{ +1, -1\}$ and the value $m_0 \in {\mathbb Z}_{<0}$ satisfies Eq.(\ref{eq:m0N}).
Let $\varepsilon $ be a real number such that $|\varepsilon |$ is sufficiently small.
By setting $A=(2N+1)Q$, $m'= m_0 -2N$, $k=j$, $C' = C $, $D'=\varepsilon $, $\eta =+1$ and $\zeta =\chi $ in Eq.(\ref{eq:solZ---Y-++}), we have
\begin{align}
& (\eta _{m_0 -2N + 3 j } , Y_{m_0 -2N + 3 j } )= (+1 , 2 j Q + \varepsilon ), \label{eq:YZNep1} \\
& (\zeta _{m_0 -2N + 3 j },Z_{m_0 -2N + 3 j  } )= ((-1)^j \chi , - j ^2 Q - C - j \varepsilon  ), \nonumber \\
& (\eta _{m_0 -2N + 3 j +1} ,Y_{m_0 -2N + 3 j +1}) = ((-1)^j \chi , (m_0 - j ^2 +  j +1) Q - C -(j+1)\varepsilon  ),  \nonumber \\
& (\zeta _{m_0 -2N + 3 j +1} , Z_{m_0 -2N + 3 j +1}) = ( +1 ,(m_0 + j +1) Q -\varepsilon  ),  \nonumber \\
& (\eta _{m_0 -2N + 3 j +2}  ,Y_{m_0 -2N + 3 j +2} )=  ((-1)^j \chi , (m_0 +j^2+ 3 j +2) Q  +C +j \varepsilon ), \nonumber \\
& (\zeta _{m_0 -2N + 3 j +2} , Z_{m_0 -2N + 3 j +2} )= ((-1)^j \chi ,(j+1)^2  Q   +C +(j+1) \varepsilon  ) . \nonumber 
\end{align}
We discuss the case $k_0=0$ with the condition $-m_0 Q - N(N-1)Q < C < (m_0 +1) Q $ and $|\varepsilon | \ll 1 $.
Then the function in Eq.(\ref{eq:YZNep1}) satisfies Eq.(\ref{eq:condZ---Y-++}) for $m'=m_0-2N$ and $K=N-1$, $Z_m<mQ$ $(m=m_0+N)$ and $Y_{m} >0$ $(m=m_0+N+1)$.
%\begin{equation}
%Z_m<mQ, \; Y_{m_0 -2N + 3 j +1} >0, \; Y_{m_0 -2N + 3 j' +2} >0, \; Y _{m_0 -2N + 3 j' +3} <0 
%\label{eq:YZNcondi}
%\end{equation}
%for $m_0 -2N  \leq m \leq m_0 +N $, $0 \leq j \leq N $ and $0 \leq j' \leq N -1$.
Hence it follows from Proposition \ref{prop:sol-A} that the function $(\eta _{m} , Y_{m} )$ for $m_0 -2N  \leq m \leq m_0 +N +1$ and  $(\zeta _{m} , Z_{m} )$ for $m_0 -2N  \leq m \leq m_0 +N +1$ in Eq.(\ref{eq:YZNep1}) satisfies p-ud PII.
If $\varepsilon =0$, then the function defined by Eq.(\ref{eq:YZNep1}) coincides with the solution ($(\eta^{(N)}(m) , Y^{(N)}(m))$, $(\zeta^{(N)}(m) , Z^{(N)}(m))$) in section \ref{sec:etaYIIT}, which has the indefinite evolution caused by $(\eta _{m_0 -2N } , Y_{m_0 -2N} )= (+1 , 0 )$ and $(\zeta _{m_0+N+1} , Z_{m_0 +N +1} ) =(+1, (m_0+N +1)Q)$.
By imposing $\varepsilon \neq 0$, the indefiniteness of the evolution may disappear.
In short, we obtained the solution of Type $-A$ for $m_0 -2N  \leq m \leq m_0 +N +1$.

We now investigate the backward evolution for $m \leq m_0-2N$.
Recall that 
\begin{align}
& (\zeta _{m_0 -2N },Z_{m_0 -2N} )= ( \chi , -C ),  \\  
& (\eta _{m_0 -2N } , Y_{m_0 -2N} )= (+1 , \varepsilon  ). \nonumber
\end{align}
If $\varepsilon <0$, it follows from $Y_{m_0 -2N } = \varepsilon < 0 $ that
\begin{align}
& (\zeta _{m_0 -2N -1} , Z_{m_0 -2N -1} )= ( -\chi , C ), \\
& (\eta _{m_0 -2N -1}  ,Y_{m_0 -2N -1} )=  ( -\chi  , m_0 Q  +C - \varepsilon  ), \nonumber \\
& (\zeta _{m_0 -2N -2} , Z_{m_0 -2N -2}) = ( +1 , m_0  Q -\varepsilon  ),  \nonumber \\
& (\eta _{m_0 -2N -2} ,Y_{m_0 -2N -2}) = (-\chi , (m_0 -1) Q - C ),  \nonumber \\
& (\zeta _{m_0 -2N -3 },Z_{m_0 -2N -3} )= (-\chi , - Q - C  + \varepsilon  ), \nonumber \\
& (\eta _{m_0 -2N -3 } , Y_{m_0 -2N -3 } )= ( +1 , -2 Q + \varepsilon  ), \nonumber
\end{align}
and the condition $Z_{m} <mQ$ and $Y_m >0$ is satisfied for $m=m_0 -2N -1, m_0 -2N -2$ and $ m_0 -2N -3$.
Then we can use Proposition \ref{prop:sol<<} (ii), and the solution for $m\leq m_0 -2N -1 $ is expressed in the form of Proposition \ref{prop:sol<<} (i) by setting $c_{m_0} =-C+ \varepsilon -(m_0+1) Q/3$, $c_{m_0 +1} =- \varepsilon +(2m_0+1) Q/3$ and $c_{m_0 +2} = C -m_0 Q/3$.
Namely if $j<0$, then we have
\begin{align}
& (\zeta _{m_0 -2N + 3 j +2} , Z_{m_0 -2N + 3 j +2} )= ( -\chi ,(j +1) Q +C   ) , \label{eq:IITepsiN1<<} \\. \nonumber 
& (\eta _{m_0 -2N + 3 j +2}  ,Y_{m_0 -2N + 3 j +2} )=  ( -\chi  , (m_0 +2j +2) Q  +C - \varepsilon  ), \nonumber \\
& (\zeta _{m_0 -2N + 3 j +1} , Z_{m_0 -2N + 3 j +1}) = ( +1 ,(m_0 +j+1) Q -\varepsilon  ),  \nonumber \\
& (\eta _{m_0 -2N + 3 j +1} ,Y_{m_0 -2N + 3 j +1}) = (-\chi , (m_0 +2j +1) Q - C ),  \nonumber \\
& (\zeta _{m_0 -2N + 3 j },Z_{m_0 -2N + 3 j  } )= (-\chi , j Q - C  + \varepsilon  ), \nonumber \\
& (\eta _{m_0 -2N + 3 j } , Y_{m_0 -2N + 3 j } )= ( +1 , 2j Q + \varepsilon  ) . \nonumber 
\end{align}
If $\varepsilon >0$, then we have $Y_{m_0 -2N } = \varepsilon >0 $ and the condition $Z_{m} <mQ$ and $Y_m >0$ is satisfied for $m=m_0 -2N +2, m_0 -2N +1$ and $ m_0 -2N$.
Then we can use Proposition \ref{prop:sol<<} (ii) and the solution for $m\leq m_0 -2N +2 $ is expressed in the form of Proposition \ref{prop:sol<<} (i) by setting $c_{m_0} =-C -(m_0+1) Q/3$, $c_{m_0 +1} =- \varepsilon +(2m_0+1) Q/3$ and $c_{m_0 +2} = \varepsilon + C -m_0 Q/3$.
Namely if $j\leq 0$, then 
\begin{align}
& (\zeta _{m_0 -2N + 3 j +2} , Z_{m_0 -2N + 3 j +2} )= ( \chi ,(j +1) Q +C + \varepsilon  ) , \label{eq:IITepsiN2<<} \\
& (\eta _{m_0 -2N + 3 j +2}  ,Y_{m_0 -2N + 3 j +2} )=  ( \chi  , (m_0 +2j +2) Q  +C  ), \nonumber \\
& (\zeta _{m_0 -2N + 3 j +1} , Z_{m_0 -2N + 3 j +1}) = ( +1 ,(m_0 +j+1) Q -\varepsilon  ),  \nonumber \\
& (\eta _{m_0 -2N + 3 j +1} ,Y_{m_0 -2N + 3 j +1}) = (\chi , (m_0 +2j +1) Q - C -\varepsilon ),  \nonumber \\
& (\zeta _{m_0 -2N + 3 j },Z_{m_0 -2N + 3 j  } )= (\chi , j Q - C  ), \nonumber \\
& (\eta _{m_0 -2N + 3 j } , Y_{m_0 -2N + 3 j } )= ( +1 , 2j Q + \varepsilon  ) , \nonumber 
\end{align}
whose amplitude functions are slightly different from Eq.(\ref{eq:IITepsiN1<<}).
In either case, we obtained the solution of Type $--$.
Recall that the Airy-type solution for $m \leq m_0 - 2N -1 $ is written as Eq.(\ref{eq:IITNsolm<<}), i.e.
\begin{equation}
(\zeta^{(N)}(m) , Z^{(N)}(m))=(+1 , (-m-2N-1)Q),
%\label{eq:IITNsolm<<} 
\end{equation}
and completely different from the solution in Eq.(\ref{eq:IITepsiN1<<}) or Eq.(\ref{eq:IITepsiN2<<}), whose initial value is perturbed.

We now investigate the forward evolution for $m \geq m_0+N+1$.
Recall that
\begin{align}
& (\eta _{m_0 + N +1} ,Y_{m_0 + N +1}) = ((-1)^N \chi , (m_0 - N ^2 + N +1) Q - C -(N+1)\varepsilon  ), \nonumber \\
& (\zeta _{m_0 +N +1} , Z_{m_0 + N +1}) = ( +1 ,(m_0 + N +1) Q -\varepsilon  ). 
\end{align}
If $\varepsilon >0$, then $Z_{m_0 +N +1} < (m_0 + N +1) Q$ and we have 
\begin{align}
& (\eta _{m_0 +N +2}  ,Y_{m_0 +N +2} )=  ((-1)^N \chi , (m_0 +N^2+ 3 N +2) Q  +C +N \varepsilon ), \nonumber \\
& (\zeta _{m_0 +N +2} , Z_{m_0 +N +2} ) \\
& = \left\{ 
\begin{array}{ll}
((-1)^N \chi ,(N+1)^2  Q   +C +(N+1) \varepsilon  ) ,&  (Y_{m_0 +N +2}>0 )\\
(-1 , -(m_0 + N +1) Q +\varepsilon ), & (Y_{m_0 +N +2}<0) .
\end{array}
\right.
 \nonumber 
\end{align}
If $\varepsilon <0$, then $Z _{m_0 +N +1} > (m_0 + N +1) Q$ and we have 
\begin{align}
& (\eta _{m_0 +N +2}  ,Y_{m_0 +N +2} )=  ((-1)^{N+1} \chi , (m_0 +N^2+ 3 N +2) Q  +C +(N+1) \varepsilon ), \nonumber \\
& (\zeta _{m_0 +N +2} , Z_{m_0 +N +2} ) \\
& = \left\{ 
\begin{array}{ll}
((-1)^{N+1} \chi ,(N+1)^2  Q   +C +(N+2) \varepsilon  ) ,&  (Y_{m_0 +N +2}>0) \\
(-1 , -(m_0 + N +1) Q +\varepsilon ), & (Y_{m_0 +N +2}<0) .
\end{array}
\right. \nonumber 
\end{align}
Therefore they are completely different from the Airy-type solution for $m_0 +N +2\leq m \leq -1 $ written as
\begin{equation}
(\zeta^{(N)}(m) , Z^{(N)}(m))=(+1 , mQ).
\end{equation}

We discuss the case $k_0 \neq 0$ with the condition $-m_0 Q - (N-k_0 )(N-k_0 +1)Q  < C < -m_0 Q- (N-k_0 +1)(N-k_0 +2)Q (< (m_0+1)Q)$.
%Then it is shown that the function in Eq.(\ref{eq:YZNep1}) satisfies the condition in Eq.(\ref{eq:YZNcondi}) for $m_0 -2N  \leq m \leq m_0 +N - 3 k_0 +1$, $0 \leq j \leq N -k_0$ and $0 \leq j' \leq N -k_0 -1$.
Then the function in Eq.(\ref{eq:YZNep1}) satisfies Eq.(\ref{eq:condZ---Y-++}) for $m'=m_0-2N$ and $K=N-k_0-1$, $Z_m<mQ$ ($m=m_0 +N-3k_0 $ and $m_0 +N-3k_0 +1 $) and $Y_{m} >0$ $(m=m_0+N-3k_0 +1)$.
%Hence it follows from Proposition \ref{prop:sol-A} that the function $(\eta _{m} , Y_{m} )$ for $m_0 -2N  \leq m \leq m_0 +N +1$ and  $(\zeta _{m} , Z_{m} )$ for $m_0 -2N  \leq m \leq m_0 +N +1$ in Eq.(\ref{eq:YZNep1}) satisfies p-ud PII.
Hence it follows from Proposition \ref{prop:sol-A} that the function $(\eta _{m} , Y_{m} )$ for $m_0 -2N  \leq m \leq m_0 +N - 3 k_0 +2$ and  $(\zeta _{m} , Z_{m} )$  for $m_0 -2N  \leq m \leq m_0 +N - 3 k_0 +1$ in Eq.(\ref{eq:YZNep1}) satisfies p-ud PII.
However we have  
\begin{align}
& (\eta _{m_0 +N - 3 k_0 +2}  ,Y_{m_0 +N - 3 k_0 +2} )= \\
& \qquad \qquad ((-1)^{N-k_0} \chi , (m_0 +(N-k_0+1)(N-k_0 +2)) Q +C +(N-k_0) \varepsilon ), \nonumber
\end{align}
and $Y_{m_0 +N - 3 k_0 +2} <0 $, namely the condition is changed.
By applying the forward evolution of p-ud PII, we have
\begin{align}
& (\zeta _{m_0 +N - 3 k_0 +2}, Z_{m_0 +N - 3 k_0 +2} )= (-1 , -(m_0 + N -k_0 +1) Q + \varepsilon   ), \label{eq:YZNep1p} \\
& (\eta _{m_0 +N - 3 k_0 +3}  ,Y_{m_0 +N - 3 k_0 +3} )= \nonumber\\
& \qquad ((-1)^{N-k_0+1} \chi , -(m_0 +(N-k_0+1)(N-k_0 )) Q -C -(N-k_0-1) \varepsilon), \nonumber
\end{align}
with the condition $ Z_{m_0 +N - 3 k_0 +2} < (m_0 +N - 3 k_0 +2 )Q $ and $Y_{m_0 +N - 3 k_0 +3}  <0 $. Then
\begin{align}
& (\zeta _{m_0 +N - 3 k_0 +3}, Z_{m_0 +N - 3 k_0 +3} )= (+1 , (m_0 + N -k_0 +1) Q - \varepsilon   ).  
\end{align}
In the case $k_0 >1 $, we apply further evolution.
Set $D'=\varepsilon $, $A=(2N+1)Q$, $m'= m_0 -2N -1$, $k=j$, $C' = 2 m_0 Q +C -3 \varepsilon $,  $\eta =+1$ and $\zeta =\chi $ in Eq.(\ref{eq:solZ---Y-++}).
Then we have
\begin{align}
& (\eta _{m_0 -2N + 3 j } , Y_{m_0 -2N + 3 j } )= ((-1)^j \chi  , (- m_0 - j( j-1) ) Q  - C -(j-2) \varepsilon  ), \nonumber \\
& (\zeta _{m_0 -2N + 3 j },Z_{m_0 -2N + 3 j  } )= (+1 , (m_0 +  j ) Q -\varepsilon  ), \label{eq:YZNep2} \\
& (\eta _{m_0 -2N + 3 j +1} ,Y_{m_0 -2N + 3 j +1}) = ((-1)^j \chi , (3 m_0 +j^2+ 3 j +1) Q +C +(j -3) \varepsilon  ),  \nonumber \\
& (\zeta _{m_0 -2N + 3 j +1} , Z_{m_0 -2N + 3 j +1}) = ( -1)^j \chi ,( 2 m_0 + (j+1)^2  )Q +C +(j-2) \varepsilon  ),  \nonumber \\
& (\eta _{m_0 -2N + 3 j +2}  ,Y_{m_0 -2N + 3 j +2} )=  ( +1 , (2 j +2) Q + \varepsilon  ), \nonumber \\
& (\zeta _{m_0 -2N + 3 j +2} , Z_{m_0 -2N + 3 j +2} )= ((-1)^j \chi ,-(2 m_0 + (j+1)^2) Q - C -(j-2)  \varepsilon  ) , \nonumber 
\end{align}
and the condition $Z_m<mQ$, $Y_{m_0 -2N + 3 j +1} >0$, $Y_{m_0 -2N + 3 j +2} <0$ and $Y_{m_0 -2N + 3 j+3 } >0$ is satisfied for $m_0 +N - 3 k_0 +3 \leq m \leq m_0 +N -1$ and $N-k_0+1  \leq j \leq N-1 $.
It is shown by the argument similar to the proof of Proposition \ref{prop:sol-A} that the functions $(\eta _{m} , Y_{m} )$ for $m_0 +N- 3 k_0 +3 \leq m \leq m_0 +N $ and  $(\zeta _{m} , Z_{m} )$  for $m_0 +N- 3 k_0 +3 \leq m \leq m_0 +N $ in Eq.(\ref{eq:YZNep2}) satisfy p-ud PII.
%follows from Proposition \ref{prop:sol-A}
If $\varepsilon =0$, then the function defined by Eqs.(\ref{eq:YZNep1}), (\ref{eq:YZNep1p}) and (\ref{eq:YZNep2}) coincides with the solution in section \ref{sec:etaYIIT}, which has the indefinite evolution caused by $(\eta _{m_0 -2N } , Y_{m_0 -2N} )= (+1 , 0 )$ and $(\zeta _{m_0+N} , Z_{m_0 +N} ) =(+1, (m_0+N)Q)$.
By imposing $\varepsilon \neq 0$, the indefiniteness of the evolution may disappear.
In the case $k_0 > 0$, we obtained the two jointed solution of type $-A$ for $m_0 -2N \leq m \leq m_0 +N $.

The solution for $m \leq m_0-2N-1$ is obtained exactly the same as Eq.(\ref{eq:IITepsiN1<<}) or Eq.(\ref{eq:IITepsiN2<<}) which is of Type $--$, and it is different from the Airy-type solution.
The solution for $m_0 +N+1 \leq m $ is also different from the Airy-type solution, although we need case classification to express the solution.

We give an example which is related with the solution in Eqs.(\ref{eq:z3}) and (\ref{eq:y3}).
Set $Q=-3$ and $N=3$ $(A=7Q)$ and choose the initial value as $( \eta_{-18} , Y_{-18} )= (+1,0- \varepsilon)$ and $(\zeta _{-18}, Z_{-18} )= (+1,29)$ $(0<4 \varepsilon <1 )$. Then the solution of p-ud PII is written as follows:
\begin{align}
( \eta_ m , Y_m ), \: (\zeta _m, Z_m)= \left\{
\begin{array}{lll}
(+1,6 -\varepsilon), & ( -1,32 -\varepsilon) & (m= -21) \\
(-1,68), & ( +1,36 +\varepsilon) & (m= -20) \\
(-1,7+\varepsilon), & ( -1,-29) & (m= -19) \\
(+1,0-\varepsilon), & (+1,29) & (m=-18) \\
(+1,62+2\varepsilon), & (+1,33+\varepsilon) & (m=-17) \\
(+1,1), & (+1,-32-\varepsilon) &  (m=-16) \\
(+1,-6-\varepsilon), & (-1,32+\varepsilon) & (m=-15)\\
(-1,62-2\varepsilon), & (+1,30+\varepsilon) & (m=-14) \\
(-1,-11-\varepsilon), & (-1,-30-\varepsilon) & (m=-13) \\
(+1,-1), & (+1,30+\varepsilon) & (m=-12) \\
(+1,46+\varepsilon), & (+1,16) & (m=-11) \\
(+1,-18-\varepsilon), & (-1,-16) & (m=-10) \\
(-1,11+\varepsilon), & ( +1, 27+\varepsilon) & (m=-9 ) \\
(+1,22-\varepsilon), & ( +1, -5-2\varepsilon) & (m=-8 ) \\
(+1,-24-\varepsilon), & ( -1, 5+2\varepsilon) & (m=-7 ) \\
(-1,29+3\varepsilon), & ( +1, 24+\varepsilon) & (m=-6 ) \\
(+1,-14-3\varepsilon), & ( -1, -24-\varepsilon) & (m=-5 ) \\
(-1,-16+2\varepsilon), & ( +1, 24+\varepsilon) & (m=-4 ) \\
(+1,19-2\varepsilon), & ( +1, -5-3\varepsilon) & (m=-3 ) \\
(+1,-36-\varepsilon), & ( -1, 5+3\varepsilon) & (m=-2 ) \\
(-1,26+4\varepsilon), & ( +1, 21+\varepsilon) & (m=-1 ) \\
(+1,-41-4\varepsilon), & ( -1, -21-\varepsilon) & (m=0 ) \\
(-1,-1+3\varepsilon), & ( +1, 21+\varepsilon) & (m=1 ) \\
(+1,-26-3\varepsilon), & ( -1, -21-\varepsilon) & (m=2 ) \\
(-1,-21+2\varepsilon), & ( +1, 21+\varepsilon) & (m=3 ) \\
(+1,-17-2\varepsilon), & ( -1, -21-\varepsilon) & (m=4 ) \\
(-1,-37+\varepsilon), & ( +1, 21+\varepsilon) & (m=5 ) \\
(+1,-14-\varepsilon), & ( -1, -21-\varepsilon) & (m=6 ) \\
(-1,-46), & ( +1, 21+\varepsilon) & (m=7) \\
(+1,-17), & ( -1, -21-\varepsilon) & (m=8 ). 
%(-1,-52), & ( +1, 21+\varepsilon) & (m=9) \\
%(+1,-23), & ( -1, -21-\varepsilon) & (m=10 ) 
\end{array}
\right.
\label{eq:pertex1}
\end{align}
The graph of the solution is written as Figure 3, where $\bullet $ (resp. $\circ $) represents the amplitute with $\zeta _m =+1 $ (resp. $\zeta _m =-1 $).
\begin{figure}
\begin{picture}(200,250)(-50,0)
\put(20,92){\circle*{3}}
\put(25,27){\circle{3}}
\put(30,85){\circle{3}}
\put(35,89){\circle*{3}}
\put(40,24){\circle{3}}
\put(45,82){\circle{3}}
\put(50,86){\circle*{3}}
\put(55,21){\circle{3}}
\put(60,79){\circle*{3}}
\put(65,83){\circle*{3}}
\put(70,18){\circle*{3}}
\put(75,82){\circle{3}}
\put(80,80){\circle*{3}}
\put(85,20){\circle{3}}
\put(90,80){\circle*{3}}
\put(95,66){\circle*{3}}
\put(100,34){\circle{3}}
\put(105,77){\circle*{3}}
\put(110,45){\circle*{3}}
\put(115,55){\circle{3}}
\put(120,74){\circle*{3}}
\put(125,26){\circle{3}}
\put(130,74){\circle*{3}}
\put(135,45){\circle*{3}}
\put(140,55){\circle{3}}
\put(145,71){\circle*{3}}
\put(150,29){\circle{3}}
\put(155,71){\circle*{3}}
\put(160,29){\circle{3}}
\put(165,71){\circle*{3}}
\put(170,29){\circle{3}}
\put(175,71){\circle*{3}}
\put(180,29){\circle{3}}
\put(185,71){\circle*{3}}
\put(190,29){\circle{3}}
\put(195,71){\circle*{3}}
\put(200,29){\circle{3}}
\put(205,71){\circle*{3}}
\put(210,29){\circle{3}}
%\put(185,71){\circle*{3}}
%\put(190,29){\circle{3}}
\put(25,50){\vector(1,0){180}}
\put(150,15){\vector(0,1){80}}
\put(100,15){\line(0,1){75}}
\put(50,15){\line(0,1){75}}
\put(152,90){{\footnotesize $Z_m$}}
\put(200,43){{\footnotesize $m$}}
\put(142,43){{\scriptsize $O$}}
\put(84,44){{\scriptsize $-10$}}
\put(34,44){{\scriptsize $-20$}}
\put(149,80){\line(1,0){2}}
\put(152,78){{\scriptsize $30$}}
\put(149,20){\line(1,0){2}}
\put(151,17){{\scriptsize $-30$}}
%\put(40,103){Perturbed solution}
%\put(40,103){{\small mostly coincide with det-type sol.}}
%\put(60,99){\line(0,-1){5}}
%\put(105,99){\line(0,-1){5}}
%\put(60,99){\line(1,0){45}}
\put(25,180){\vector(1,0){180}}
\put(150,110){\vector(0,1){145}}
\put(100,110){\line(0,1){145}}
\put(50,110){\line(0,1){145}}
\put(152,250){{\footnotesize $Y_m$}}
\put(200,173){{\footnotesize $m$}}
\put(142,173){{\scriptsize $O$}}
\put(84,174){{\scriptsize $-10$}}
\put(34,174){{\scriptsize $-20$}}
\put(149,210){\line(1,0){2}}
\put(152,208){{\scriptsize $30$}}
\put(149,240){\line(1,0){2}}
\put(152,238){{\scriptsize $60$}}
\put(149,150){\line(1,0){2}}
\put(151,147){{\scriptsize $-30$}}
\put(149,120){\line(1,0){2}}
\put(151,117){{\scriptsize $-60$}}
\put(20,260){\circle{3}}
\put(25,199){\circle{3}}
\put(30,192){\circle*{3}}
\put(35,254){\circle{3}}
\put(40,193){\circle{3}}
\put(45,186){\circle*{3}}
\put(50,248){\circle{3}}
\put(55,187){\circle{3}}
\put(60,180){\circle*{3}}
\put(65,242){\circle*{3}}
\put(70,181){\circle*{3}}
\put(75,174){\circle*{3}}
\put(80,242){\circle{3}}
\put(85,169){\circle{3}}
\put(90,179){\circle*{3}}
\put(95,226){\circle*{3}}
\put(100,162){\circle*{3}}
\put(105,191){\circle{3}}
\put(110,202){\circle*{3}}
\put(115,156){\circle*{3}}
\put(120,209){\circle{3}}
\put(125,166){\circle*{3}}
\put(130,164){\circle{3}}
\put(135,199){\circle*{3}}
\put(140,144){\circle*{3}}
\put(145,206){\circle{3}}
\put(150,139){\circle*{3}}
\put(155,179){\circle{3}}
\put(160,154){\circle*{3}}
\put(165,159){\circle{3}}
\put(170,163){\circle*{3}}
\put(175,143){\circle{3}}
\put(180,166){\circle*{3}}
\put(185,134){\circle{3}}
\put(190,163){\circle*{3}}
\put(195,128){\circle{3}}
\put(200,157){\circle*{3}}
\put(205,122){\circle{3}}
\put(210,151){\circle*{3}}
\put(55,9){\line(0,1){5}}
\put(55,9){\line(-1,0){40}}
\put(10,0){{\small Type $--$}}
\put(57,9){\line(0,1){5}}
\put(68,9){\line(0,1){5}}
\put(77,9){\line(0,1){5}}
\put(105,9){\line(0,1){5}}
\put(57,9){\line(1,0){11}}
\put(77,9){\line(1,0){28}}
\put(60,0){{\small Type $-A$}}
\put(145,9){\line(0,1){5}}
\put(145,9){\line(1,0){35}}
\put(145,0){{\small Type $+A$}}
\put(180,9){\line(0,1){5}}
\put(185,9){\line(0,1){5}}
\put(185,9){\line(1,0){30}}
\put(190,0){{\small Type $++$}}
\put(150,50){\line(5,-3){60}}
\put(150,50){\line(-5,3){80}}
\put(73,95){{\footnotesize $Z_m=mQ$}}
\end{picture}
\caption{The solution associated with Eq.(\ref{eq:pertex1})}
\end{figure}
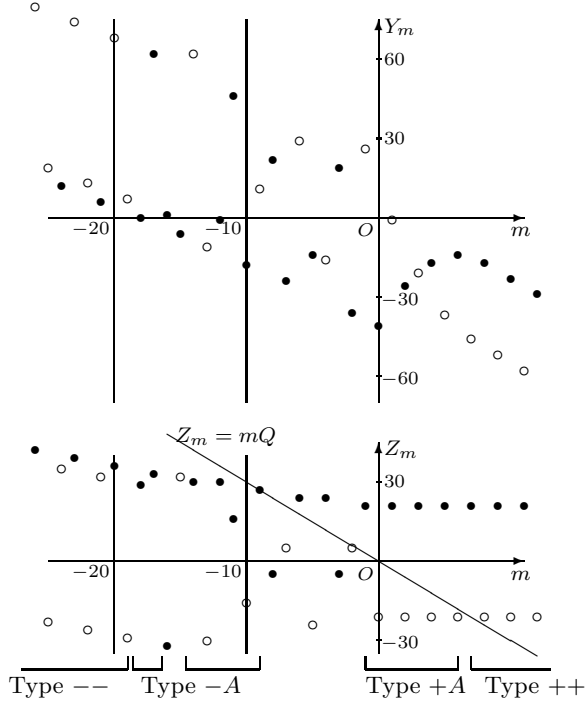
If we impose $\varepsilon =0$, then the function coincides with the Airy-type solution given in Eqs.(\ref{eq:z3}) and (\ref{eq:y3}) for $-18\leq m\leq -9$, which is of Type $-A$.
The values $( \eta_{-18} , Y_{-18} ) = (+1,0-\varepsilon ) $ and $(\zeta _{-9}, Z_{-9})= ( +1, 27+\varepsilon )$ in the case $\varepsilon =0$ show the indefinite evolution, although it is the unique evolution in the case $\varepsilon  \neq 0$.
%Assume that $0< \varepsilon  <1/4 $.
If $0< \varepsilon  <1/4 $ and $m \leq -19 $, then the solution is written in the form of Proposition \ref{prop:sol<<} (i.e.~of Type $--$), namely
\begin{align}
& ( \eta_ m , Y_m ), \: (\zeta _m, Z_m)= \\
&  \left\{
\begin{array}{lll}
(+1,6j -36 -\varepsilon ), & ( -1,3j +11 -\varepsilon ) & (m= -3j ) \\
(-1,6j + 26 ), & ( +1,3j +15 +\varepsilon ) & (m= -3j +1 ) \\
(-1,6j -35 +\varepsilon ), & ( -1,3j -50 ) & (m= -3j +2) 
\end{array}
\right. \nonumber
\end{align}
for $j \geq 7$, and it is quite different from the Airy-type solution written as $( \zeta_{m} , Z_{m} ) = (+1, 3m+21) $ in Eq.(\ref{eq:z3}).
If $-18\leq m \leq  -9$, then the solution is written in the form of Eqs.(\ref{eq:YZNep1}) and (\ref{eq:YZNep2}), and it is of Type $-A$.
In the case $0\leq m\leq 6 $, the solution is in the form of Proposition \ref{prop:sol+A} (i.e.~of Type $+A$).
% of Type $+A$, i.e. it is written in the form of Eqs.(\ref{eq:condZ-+Y--}) and (\ref{eq:solZ-+Y--}).
In the case $m\geq 7$, the solution is written in the form of Proposition \ref{prop:sol>>} (i.e.~of Type $++$), namely
\begin{align}
& ( \eta_ m , Y_m ), \: (\zeta _m, Z_m)=  \left\{
\begin{array}{lll}
(-1,-6j -22 ), & ( +1,21+\varepsilon ) & (m= 2j -1 ) \\
(+1,-6j +7 ), & ( -1,-21-\varepsilon  ) & (m= 2j ) 
\end{array}
\right.
\end{align}
for $j \geq 4$, and it is quite different from the Airy-type solution in Eq.(\ref{eq:z3}).
The function $(\zeta _m, Z_m )$ for $m \gg 1$ is oscillating and the solutions to the ultradiscrete Airy equation (see Eqs.(41) and (42) in \cite{IKMMS}) are also oscillating, but their exact forms are different.

%If the initial value $Y_{-18} $ is positive, then the solution is written in a slightly different form.
We express the solution whose initial value $Y_{-18}$ is positive and small.
We choose the initial value as $( \eta_{-18} , Y_{-18} )= (+1,0+\varepsilon )$ and $(\zeta _{-18}, Z_{-18} )= (+1,29)$ $(0<3 \varepsilon  <1 )$. Then the solution of p-ud PII is written as follows:
\begin{align}
& ( \eta_ m , Y_m ), \: (\zeta _m, Z_m)= \label{eq:pertex2} \\
& \left\{
\begin{array}{lll}
(+1,6j -36 +\varepsilon ), & ( +1,3j +11 ) & (m= -3j , \: j \geq 6) \\
(+1,6j + 26 -\varepsilon ), & ( +1,3j +15 -\varepsilon ) & (m= -3j +1 , \: j \geq 6) \\
(+1,6j -35 ), & ( +1,3j -50 +\varepsilon  ) & (m= -3j +2, \: j \geq 6) \\
(+1,-6+\varepsilon ), & (-1,32-\varepsilon ) & (m=-15)\\
(-1,62-2\varepsilon ), & (+1,30-\varepsilon ) & (m=-14) \\
(-1,-11+\varepsilon ), & (-1,-30+\varepsilon ) & (m=-13) \\
(+1,-1), & (+1,30-\varepsilon ) & (m=-12) \\
(+1,46-\varepsilon ), & (+1,16) & (m=-11) \\
(+1,-18+\varepsilon ), & (-1,-16) & (m=-10) \\
(-1,11-\varepsilon ), & (-1, 27-\varepsilon ) & (m=-9 ) \\
(-1,22), & (+1, -5+\varepsilon ) & (m=-8 ) \\
(+1,-24+\varepsilon ), & ( -1, 5-\varepsilon ) & (m=-7 ) \\
(+1,29-2\varepsilon ), & ( +1, 24-\varepsilon ) & (m=-6 ) \\
(-1,-14+2\varepsilon ), & ( -1, -24+\varepsilon ) & (m=-5 ) \\
(+1,-16-\varepsilon ), & ( +1, 24-\varepsilon ) & (m=-4 ) \\
(-1,19+\varepsilon ), & ( -1, -5+2\varepsilon ) & (m=-3 ) \\
(+1,-36+\varepsilon ), & ( +1, 5-2\varepsilon ) & (m=-2 ) \\
(+1,26-3\varepsilon ), & ( +1, 21-\varepsilon ) & (m=-1 ) \\
(-1,-41+3\varepsilon ), & ( -1, -21+\varepsilon ) & (m=0 ) \\
(+1,-1-2\varepsilon ), & ( +1, 21-\varepsilon ) & (m=1 ) \\
(-1,-26+2\varepsilon ), & ( -1, -21+\varepsilon ) & (m=2 ) \\
(+1,-21-\varepsilon ), & ( +1, 21-\varepsilon ) & (m=3 ) \\
(-1,-17+\varepsilon ), & ( -1, -21+\varepsilon ) & (m=4 ) \\
(+1,-37), & ( +1, 21-\varepsilon ) & (m=5 ) \\
(-1,-14), & ( -1, -21+\varepsilon ) & (m=6 ) \\
(+1,-6j -22 +\varepsilon ), & ( +1,21-\varepsilon ) & (m= 2j -1 , \: j \geq 4) \\
(-1,-6j +7 -\varepsilon ), & ( -1,-21+\varepsilon  ) & (m= 2j , \: j \geq 4) .
\end{array}
\right. \nonumber
\end{align}
Although the amplitude functions shift slightly from Eq.(\ref{eq:pertex1}), the sign functions change essentially.

Let us investigate the solutions of p-ud PII which are related with the function $(\zeta^{(M)}(m) , Z^{(M)}(m))$ ($M \in {\mathbb Z}_{\leq -1}$) in section \ref{sec:etaYIIT}.
We assume that the constant $A$ in the p-ud PII is written as $A=(2M+1)Q$ ($M \in {\mathbb Z}_{\leq -1}$), $k_0 \in \{ 0, 1 , \cdots , -M-1 \}$, $\chi \in \{ +1, -1\}$ and the value $m_0 \in {\mathbb Z}_{<0}$ satisfies Eq.(\ref{eq:m0M}).
Let $\varepsilon $ be a real number such that $|\varepsilon |$ is sufficiently small.

We set $A=(2M+1)Q$, $m'= m_0 +M+1$, $k=j-1$, $C' = (M +1) Q +C $, $D'=\varepsilon $, $\eta =-\chi $, $\zeta =+1$ in Eq.(\ref{eq:solZ+--Y+++}).
Then we have
\begin{align}
& (\zeta _{m_0 +M + 3 j -2},Z_{m_0 +M + 3 j -2 } )= ( +1 , (m_0 +M+j) Q+\varepsilon  ), \label{eq:ZYMep1} \\
& (\eta _{m_0 +M + 3 j -1} ,Y_{m_0 +M + 3 j -1}) = ((-1)^{j-1} \chi , (m_0 +2M +j(j+1)) Q-(j-1) \varepsilon  +C ),  \nonumber \\
& (\zeta _{m_0 +M + 3 j -1} , Z_{m_0 +M + 3 j -1}) = ( (-1)^{j-1} \chi  ,(j^2 +M) Q -j \varepsilon  +C ),  \nonumber \\
& (\eta _{m_0 +M + 3 j } , Y_{m_0 +M + 3 j } )= ( +1 , 2 (j +M) Q -\varepsilon  ),\nonumber  \\
& (\zeta _{m_0 +M + 3 j },Z_{m_0 +M + 3 j  } )= ((-1)^{j-1} \chi , (-j^2 + 2j +M ) Q + (j-1) \varepsilon  -C ), \nonumber \\
& (\eta _{m_0 +M + 3 j +1} ,Y_{m_0 +M + 3 j +1}) = ((-1)^{j-1} \chi  , (m_0 +2M -j^2+3j +1)Q+ j \varepsilon -C ).  \nonumber 
\end{align}
We discuss the case $k_0=0$ with the condition $(-m_0- (M+1)M)Q < C < (m_0 +1) Q $.
Then the function in Eq.(\ref{eq:ZYMep1}) satisfies Eq.(\ref{eq:condZ+--Y+++}) for $m'=m_0 +M+1$ and $K=-M-2$, $Y_{m} >0$ $(m=m_0 -2M-1)$ and $Z_m<mQ$ $(m=m_0 -2M-1)$.
%the conditions
%\begin{align}
%& Y_m >0 , \; Z_{m_0 +M + 3 j -1} < (m_0 +M + 3 j -1)Q, \label{eq:YZMcondi} \\
%& Z_{m_0 +M + 3 j' } < (m_0 +M + 3 j' )Q, \; Z _{m_0 +M + 3 j' +1} > (m_0 +M + 3 j' + 1)Q \nonumber
%\end{align}
%for $m_0 +M +2 \leq m \leq m_0 -2M - 1 $, $1 \leq j \leq -M $ and $1 \leq j' \leq -M  -1$.
Hence  it follows from Proposition \ref{prop:sol-B} that the function $(\eta _{m} , Y_{m} )$ for $m_0 +M +2 \leq m \leq m_0 -2M $ and  $(\zeta _{m} , Z_{m} )$ for $m_0 +M +1 \leq m \leq m_0 -2M-1$ in Eq.(\ref{eq:ZYMep1}) satisfies p-ud PII.
If $\varepsilon =0$, then the function defined by Eq.(\ref{eq:ZYMep1}) coincides with the solution in Propositions \ref{prop:zpum_0_M} (ii) and \ref{prop:::} (i), which has the indefinite evolution caused by $(\zeta _{m_0+M+1} , Z_{m_0 +M+1} ) =(+1, (m_0+M+1)Q)$ and $(\eta _{m_0 -2M } , Y_{m_0 -2M} )= (+1 , 0 )$.
By imposing $\varepsilon \neq 0$, the indefiniteness of the evolution may disappear.
In short, we obtained the solution of Type $-B$ for $m_0 +M+2  \leq m \leq m_0 -2M-1$.

We investigate the backward evolution for $m \leq m_0+M+1$.
Recall that 
\begin{align}
& (\eta _{m_0 +M +2} , Y_{m_0 +M +2} ) = ( \chi , (m_0 +2M +2) Q +C ) , \\
& (\zeta _{m_0 +M +1} , Z_{m_0 +M +1} ) = ( +1 , (m_0 +M+1) Q+ \varepsilon ) . \nonumber 
\end{align}
If $\varepsilon >0$, then it follows from $Z_{m_0 +M +1} > (m_0 +M+1) Q$ that
\begin{align}
& (\eta _{m_0 +M + 1} , Y_{m_0 +M + 1} )= ( -\chi ,(m_0 +2M -1) Q -C  ), \\
& (\zeta _{m_0 +M },Z_{m_0 +M } )= ( -\chi  , M Q - \varepsilon  -C  ), \nonumber \\
& (\eta _{m_0 +M } , Y_{m_0 +M } )= ( +1 , 2 M Q -\varepsilon  ),\nonumber  \\
& (\zeta _{m_0 +M -1} , Z_{m_0 +M -1}) = ( -\chi  ,M Q +C ),  \nonumber \\
& (\eta _{m_0 +M -1} ,Y_{m_0 +M -1}) = ( -\chi , (m_0 +2M ) Q + \varepsilon  +C),  \nonumber \\
& (\zeta _{m_0 +M -2},Z_{m_0 +M -2 } )= ( +1 , (m_0 +M) Q+ \varepsilon  ), \nonumber 
\end{align}
and the condition $Z_{m} <mQ$ ($m=m_0 +M , m_0 +M-1$ and $ m_0 +M-2$) and $Y_m >0$ ($m=m_0 +M+1, m_0 +M+2$ and $ m_0 +M+ 3$) is satisfied.
Then we can use Proposition \ref{prop:sol<<} (ii), and the solution is expressed in the form of Proposition \ref{prop:sol<<} (i) by setting $c_{m_0} =-C - \varepsilon -(m_0+1) Q/3$, $c_{m_0 +1} = \varepsilon +(2m_0+1) Q/3$ and $c_{m_0 +2} = C -m_0 Q/3$.
Namely if $j<0$, then 
\begin{align}
& (\zeta _{m_0 +M + 3 j },Z_{m_0 +M + 3 j  } )= ( -\chi  , (M +j) Q - \varepsilon  -C  ), \label{eq:IgaepsiM1<<} \\
& (\eta _{m_0 +M + 3 j } , Y_{m_0 +M + 3 j } )= ( +1 , 2 (M +j ) Q -\varepsilon  ),\nonumber  \\
& (\zeta _{m_0 +M + 3 j -1} , Z_{m_0 +M + 3 j -1}) = ( -\chi  ,(M+j) Q +C ),  \nonumber \\
& (\eta _{m_0 +M + 3 j -1} ,Y_{m_0 +M + 3 j -1}) = ( -\chi , (m_0 +2M +2j ) Q + \varepsilon  +C),  \nonumber \\
& (\zeta _{m_0 +M + 3 j -2},Z_{m_0 +M + 3 j -2 } )= ( +1 , (m_0 +M+j) Q+ \varepsilon  ), \nonumber \\
& (\eta _{m_0 +M + 3 j -2} , Y_{m_0 +M + 3 j -2} )= ( -\chi ,(m_0 +2M +2j -1) Q -C  ) . \nonumber 
\end{align}
If $\varepsilon <0$, then we have $Z_{m_0 +M +1} < (m_0 +M+1) Q$ and the condition $Z_{m} <mQ$ ($m=m_0 +M +3, m_0 +M+2$ and $m_0 +M +1$) and $Y_m >0$ ($m=m_0 +M+4, m_0 +M+3$ and $m_0 +M+ 2$) is satisfied.
Then we can use Proposition \ref{prop:sol<<} (ii) and the solution is expressed in the form of Proposition \ref{prop:sol<<} (i) by setting $c_{m_0} =-C -(m_0+1) Q/3$, $c_{m_0 +1} = \varepsilon +(2m_0+1) Q/3$ and $c_{m_0 +2} = C  - \varepsilon-m_0 Q/3$, i.e. we have 
\begin{align}
& (\zeta _{m_0 +M + 3 j },Z_{m_0 +M + 3 j  } )= ( \chi  , (M +j) Q  -C  ) , \label{eq:IgaepsiM2<<} \\
& (\eta _{m_0 +M + 3 j } , Y_{m_0 +M + 3 j } )= ( +1 , 2 (M +j ) Q -\varepsilon  ),\nonumber  \\
& (\zeta _{m_0 +M + 3 j -1} , Z_{m_0 +M + 3 j -1}) = ( \chi  ,(M +j) Q - \varepsilon  +C  ),  \nonumber \\
& (\eta _{m_0 +M + 3 j -1} ,Y_{m_0 +M + 3 j -1}) = ( \chi , (m_0 +2M +2j ) Q +C ),  \nonumber \\
& (\zeta _{m_0 +M + 3 j -2},Z_{m_0 +M + 3 j -2 } )= ( +1 , (m_0 +M+j) Q+ \varepsilon  ), \nonumber \\
& (\eta _{m_0 +M + 3 j -2} , Y_{m_0 +M + 3 j -2} )= ( \chi ,(m_0 +2M +2j -1 ) Q + \varepsilon -C )  \nonumber 
\end{align} 
for $j \leq 0$.
In either case, we obtained the solution of Type $--$.
Recall that the Airy-type solution for $m \leq m_0 +M $ is written as Eq.(\ref{eq:IgaMsolm<<}) and completely different from the solution in Eq.(\ref{eq:IgaepsiM1<<}) or Eq.(\ref{eq:IgaepsiM2<<}).

We now investigate the evolution for $m \geq m_0-2M$.
By recalling the values $(\zeta _{m_0 -2M -1} , Z_{m_0 -2M -1}) $ and $(\eta _{m_0 -2M } , Y_{m_0 -2M } ) $, we have 
\begin{align}
& (\zeta _{m_0 -2M } , Z_{m_0 -2M }) = \left\{ 
\begin{array}{ll}
( (-1)^{M} \chi  ,-(M^2 +M) Q +M \varepsilon - C ), &  \varepsilon >0 , \\
( (-1)^{M-1} \chi  ,-(M^2 +M) Q +(M -1) \varepsilon - C ), &  \varepsilon <0 .
\end{array}
\right.
\end{align}
Therefore it is completely different from the Airy-type solution for $m_0 -2M \leq  m \leq M+1 $ written as $(\zeta^{(M)}(m) , Z^{(M)}(m))=(+1 , (-m -2M-1) Q)$.

Let us consider the case $k_0 \neq 0$ with the condition $ (-m_0- (M+k_0+1)(M+k_0))Q  \leq C < (-m_0 - (M+k_0)(M+k_0-1))Q  $.
Then the function in Eq.(\ref{eq:ZYMep1}) satisfies Eq.(\ref{eq:condZ+--Y+++}) for $m'=m_0 +M+1$ and $K=-M-k_0-2$, $Y_{m} >0$ $(m=m_0 -2M -3k_0 -1$ and $m_0 -2M -3k_0 )$ and $Z_m<mQ$ $(m=m_0 -2M -3k_0 -1)$.
%the conditions
%\begin{align}
%& Y_m >0 , \; Z_{m_0 +M + 3 j -1} < (m_0 +M + 3 j -1)Q, \label{eq:YZMcondi} \\
%& Z_{m_0 +M + 3 j' } < (m_0 +M + 3 j' )Q, \; Z _{m_0 +M + 3 j' +1} > (m_0 +M + 3 j' + 1)Q \nonumber
%\end{align}
%for $m_0 +M +2 \leq m \leq m_0 -2M - 1 $, $1 \leq j \leq -M $ and $1 \leq j' \leq -M  -1$.
%Hence it follows from Proposition \ref{prop:sol-B} that the function $(\eta _{m} , Y_{m} )$ for $m_0 +M +2 \leq m \leq m_0 -2M $ and  $(\zeta _{m} , Z_{m} )$ for $m_0 +M +1 \leq m \leq m_0 -2M-1$ in Eq.(\ref{eq:ZYMep1}) satisfies p-ud PII.
%If $\varepsilon =0$, then the function defined by Eq.(\ref{eq:ZYMep1}) coincides with the solution in Propositions \ref{prop:zpum_0_M} (ii) and \ref{prop:::} (i), which has the indefinite evolution caused by $(\zeta _{m_0+M+1} , Z_{m_0 +M+1} ) =(+1, (m_0+M+1)Q)$ and $(\eta _{m_0 -2M } , Y_{m_0 -2M} )= (+1 , 0 )$.

%The condition in Eq.(\ref{eq:YZMcondi}) is satisfied for $m_0 +M +2 \leq m \leq m_0 -2M - 3 k_0 $, $1 \leq j \leq -M -k_0$ and $1 \leq j' \leq -M -k_0 -1$.
Hence it follows from Proposition \ref{prop:sol-B} that the function $(\eta _{m} , Y_{m} )$ for $m_0 +M +2 \leq m \leq m_0 -2M - 3 k_0 $ and  $(\zeta _{m} , Z_{m} )$  for $m_0 +M +1 \leq m \leq m_0 -2M - 3 k_0 $ in Eq.(\ref{eq:ZYMep1}) satisfies p-ud PII.
We have $Z_{m_0 -2M -3k_0} -(m_0 -2M -3k_0 )Q = -C -m_0 Q +(M+k_0 ) (M+k_0-1)Q + (j-1) \varepsilon >0$, and the condition is changed.
By applying the forward evolution of p-ud PII, we have 
\begin{align}
& (\eta _{m_0 -2M -3k_0 +1},Y_{m_0 -2M-3k_0+1} )= (-1 ,(2m_0 - 2M -4 k_0 +1) Q +\varepsilon ), \label{eq:ZYMep1p} \\
& (\zeta _{m_0 -2M -3k_0 +1},Z_{m_0 -2M-3k_0+1} )= \nonumber \\ 
& \qquad ((-1)^{-M-k_0} , ( 2m_0  +M + (-M - k_0 +1)^2) Q + (M+ k_0 +2) \varepsilon + C  ), \nonumber \\
& (\eta _{m_0 -2M -3k_0 +2},Y_{m_0 -2M-3k_0+2} )= (+1 , ( -2k_0 +2) Q - \varepsilon ) \nonumber 
\end{align}
with the condition $Z_{m_0 -2M-3k_0+1} > (m_0 -2M-3k_0+1) Q$ and $Y_{m} >0$ $(m=m_0 -2M -3k_0 +1$ and $ m_0 -2M -3k_0 +2)$.
To obtain further evolution, we set $D'= \varepsilon $, $A=(2M+1)Q$, $m'= m_0 +M$, $k=j-1$, $C' = (2 m_0 + M +1) Q +C +3 \varepsilon  $, $\eta =-\chi $, $\zeta =+1$ in Eq.(\ref{eq:solZ+--Y+++}). Then we have
\begin{align}
& (\zeta _{m_0 +M + 3 j -2},Z_{m_0 +M + 3 j -2 } )= ( (-1)^{j-1} \chi , (2 m_0 + M +j^2 ) Q -(j-3)  \varepsilon  +C ),  \nonumber  \\
& (\eta _{m_0 +M + 3 j -1} ,Y_{m_0 +M + 3 j -1}) = ( +1 , 2(M+ j) Q - \varepsilon  ), \label{eq:ZYMep2} \\
& (\zeta _{m_0 +M + 3 j -1} , Z_{m_0 +M + 3 j -1}) = ( (-1)^{j-1} \chi  ,(-2 m_0 + M- j^2 + 2j )Q+(j-4) \varepsilon  - C ),  \nonumber \\
& (\eta _{m_0 +M + 3 j } , Y_{m_0 +M + 3 j } )= ( (-1)^{j-1} , (-m_0 +2M -j^2 +3j )Q+ (j-3) \varepsilon  -C ),\nonumber  \\
& (\zeta _{m_0 +M + 3 j },Z_{m_0 +M + 3 j  } )= ( +1 , (m_0 +M +j) Q+ \varepsilon  ), \nonumber \\
& (\eta _{m_0 +M + 3 j +1} ,Y_{m_0 +M + 3 j +1}) = ((-1)^{j-1} \chi  , (3m_0 +2M + j^2 +3j +1) Q-(j-3)  \varepsilon  +C).  \nonumber 
\end{align}
The condition $Y_m >0 $, $Z_{m_0 +M + 3 j -1} < (m_0 +M + 3 j -1)Q$, $Z_{m_0 +M + 3 j } > (m_0 +M + 3 j )Q$ and $Z _{m_0 +M + 3 j +1} < (m_0 +M + 3 j + 1)Q $ is satisfied for $m_0 -2M-3k_0+2 \leq m \leq m_0 -2M $ and $-M -k_0+ 1 \leq j \leq -M -1$.
Hence the functions $(\eta _{m} , Y_{m} )$ for $m_0 -2M -3k_0 +2 \leq m \leq m_0 -2M - 1$ and  $(\zeta _{m} , Z_{m} )$  for $m_0 -2M -3k_0 +1 \leq m \leq m_0 -2M -2$ in Eq.(\ref{eq:ZYMep2}) satisfy p-ud PII.
Note that, if $\varepsilon =0$, then the function defined by Eqs.(\ref{eq:ZYMep1}), (\ref{eq:ZYMep1p}) and (\ref{eq:ZYMep2}) coincides with the solution $(\eta^{(M)}(m) , Y^{(M)}(m))$, $(\zeta^{(M)}(m) , Z^{(M)}(m))$ in section \ref{sec:etaYIIT}, which has the indefinite evolution caused by $(\zeta _{m_0+M+1} , Z_{m_0 +M+1} ) =(+1, (m_0+M+1)Q)$ and $(\eta _{m_0 -2M-1 } , Y_{m_0 -2M-1} )= (+1 , 0 )$.
If $\varepsilon \neq 0$, then the indefiniteness of the evolution may disappear.
In the case $k_0 >0$, we obtained the two jointed solution of type $-B$ for $m_0+M+1 \leq m \leq m_0 -2M-1  $.

The solution for $m \leq m_0+M+1$ is described exactly the same as Eqs.(\ref{eq:IgaepsiM1<<}) or (\ref{eq:IgaepsiM2<<}) which is of Type $--$, and it is different from the Airy-type solution.
The solution for $m_0 -2M \leq m $ is obtained similarly to the case $k_0=0$, and it is also different from the Airy-type solution.

We give an example which is related with the solution in Eqs.(\ref{eq:ZM-3}) and (\ref{eq:YM-3}).
Set $Q=-2$ and $M=-3$ $(A=-5Q)$ and choose the initial value as  $(\zeta _{-12}, Z_{-12} )= (+1,24+\varepsilon )$ and $( \eta_{-11} , Y_{-11} )= (+1,18)$ $(0<3 \varepsilon  <1 )$. Then the solution of p-ud PII is written as follows:
\begin{align}
( \eta_ m , Y_m ), \: (\zeta _m, Z_m)= \left\{
\begin{array}{lll}
(-1,22+\varepsilon ), & (-1,-4) & (m=-14) \\
(+1,12-\varepsilon ), & (-1,16-\varepsilon ) & (m=-13) \\
(-1,40), & (+1,24+\varepsilon ) & (m=-12) \\
(+1,18), & (+1,-6-\varepsilon ) & (m=-11) \\
(+1,8-\varepsilon ), & (+1,14) & (m=-10) \\
(+1,36+\varepsilon ), & (+1, 22+\varepsilon ) & (m=-9 ) \\
(-1,10-\varepsilon ), & (-1, -12-2\varepsilon ) & (m=-8 ) \\
(+1,4-\varepsilon ), & ( -1, 16+\varepsilon ) & (m=-7 ) \\
(-1,34+\varepsilon ), & ( +1, 18) & (m=-6 ) \\
(+1,0-\varepsilon ), & (-1, -18) & (m=-5 ) \\
(-1,2+\varepsilon ), & (+1, 20+\varepsilon ) & (m=-4 ) \\
(+1,24-\varepsilon ), & ( +1, 4-2\varepsilon ) & (m=-3 ) \\
(+1,-4-\varepsilon ), & ( -1, -4+2\varepsilon ) & (m=-2 ) \\
(-1,14+3\varepsilon ), & ( +1, 18+\varepsilon ) & (m=-1 ) \\
(+1,0-3\varepsilon ), & (+1, -18-\varepsilon ) & (m=0 ) \\
(+1,-8+2\varepsilon ), & ( -1, 18+\varepsilon ) & (m=1 ) \\
(-1,14-2\varepsilon ), & (+1, -4-3\varepsilon ) & (m=2 ) \\
(+1,-12-\varepsilon ), & ( -1, 4+3\varepsilon ) & (m=3 ) \\
(-1,10+\varepsilon ), & ( +1, 6-2\varepsilon ) & (m=4 ) \\
(+1,-16-\varepsilon ), & ( -1, -6+2\varepsilon ) & (m=5 ) \\
(-1,6+\varepsilon ), & ( +1, 12-\varepsilon ) & (m=6 ) \\
(+1,-20-\varepsilon ), & (-1, -12+\varepsilon ) & (m=7) \\
(-1,2+\varepsilon ), & ( +1, 14) & (m=8 ) \\
(+1,-24-\varepsilon ), & ( -1, -14) & (m=9) \\
(-1,-2+\varepsilon ), & ( +1, 14) & (m=10 ) .
\end{array}
\right.\label{eq:pertex3}
\end{align}
The graph of the solution is written as Figure 4.
%, where $\bullet $ (resp. $\circ $) represents the amplitute with $\zeta _m =+1 $ (resp. $\zeta _m =-1 $).
The function $( \eta_{m} , Y_{m} )$ (resp. $( \zeta_{m} , Z_{m} )$) for $-11\leq m\leq -5$ (resp. $-12\leq m\leq -6$) is of Type $-B$,
written in the form of Eqs.(\ref{eq:condZ++Y+-}) and (\ref{eq:solZ++Y+-}), 
and they coincide with the Airy-type solution given in Eq.(\ref{eq:YM-3}) (resp. Eq.(\ref{eq:ZM-3})) with the additional condition $\varepsilon = 0$.
The values $( \zeta_{-12} , Z_{-12} ) = (+1,24+\varepsilon ) $ and $(\eta _{-5}, Y_{-5})= ( +1,0-\varepsilon )$ for the case $\varepsilon =0$ show the indefinite evolution, although the evolution is unique in the case $\varepsilon  \neq 0$.
Assume that $0<3 \varepsilon  <1 $.
In the case $m \geq -12 $, the solution is written in the form of Proposition \ref{prop:sol<<} (i.e.~of Type $--$), namely
\begin{align}
& ( \eta_ m , Y_m ), \: (\zeta _m, Z_m)= \\
& \left\{
\begin{array}{lll}
(-1,4j +24 ), & ( +1,2j +16 +\varepsilon ) & (m= -3j , \: j \geq 4) \\
(-1,4j + 2 +\varepsilon  ), & ( -1,2j -14 ) & (m= -3j +1 , \: j \geq 5) \\
(+1,4j -8 -\varepsilon ), & ( -1,2j +6 -\varepsilon  ) & (m= -3j +2, \: j \geq 5) 
\end{array}
\right.  \nonumber
\end{align}
%for $m \leq -12$.
and it is quite different from the Airy-type solution written as $( \zeta_{m} , Z_{m} ) = (+1, -2m) $ in Eq.(\ref{eq:ZM-3}).
In the case $3\leq m\leq 9 $, the solution is written in the form of Proposition \ref{prop:sol+B} (i.e.~of Type $+B$).
%Eqs.(\ref{eq:condZ++Y+-}), (\ref{eq:solZ++Y+-}).
In the case $m\geq 9$, the solution is is written in the form of Proposition \ref{prop:sol>>} (i.e.~of Type $++$), namely
\begin{align}
& ( \eta_ m , Y_m ), \: (\zeta _m, Z_m)= \\
& \left\{
\begin{array}{lll}
(+1,-4j -4 -\varepsilon ), & ( -1,-14) & (m= 2j -1 , \: j \geq 5) \\
(-1,-4j +18 +\varepsilon ), & ( +1,14) & (m= 2j , \: j \geq 5) 
\end{array}
\right.  \nonumber
\end{align}
and it is quite different from the Airy-type solution written as $( \zeta_{m} , Z_{m} ) = (+1, (- m-2M-1)Q) $.
\begin{figure}
\begin{center}
\begin{picture}(230,200)(30,0)
\put(30,50){\circle{3}}
\put(35,70){\circle{3}}
\put(40,78){\circle*{3}}
\put(45,48){\circle{3}}
\put(50,68){\circle{3}}
\put(55,76){\circle*{3}}
\put(60,46){\circle{3}}
\put(65,66){\circle{3}}
\put(70,74){\circle*{3}}
\put(75,44){\circle*{3}}
\put(80,64){\circle*{3}}
\put(85,72){\circle*{3}}
\put(90,38){\circle{3}}
\put(95,66){\circle{3}}
\put(100,68){\circle*{3}}
\put(105,32){\circle{3}}
\put(110,70){\circle*{3}}
\put(115,54){\circle*{3}}
\put(120,46){\circle{3}}
\put(125,68){\circle*{3}}
\put(130,32){\circle*{3}}
\put(135,68){\circle{3}}
\put(140,46){\circle*{3}}
\put(145,54){\circle{3}}
\put(150,56){\circle*{3}}
\put(155,44){\circle{3}}
\put(160,62){\circle*{3}}
\put(165,38){\circle{3}}
\put(170,64){\circle*{3}}
\put(175,36){\circle{3}}
\put(180,64){\circle*{3}}
\put(185,36){\circle{3}}
\put(190,64){\circle*{3}}
\put(195,36){\circle{3}}
\put(200,64){\circle*{3}}
\put(25,50){\vector(1,0){170}}
\put(130,15){\vector(0,1){80}}
\put(80,15){\line(0,1){80}}
\put(180,15){\line(0,1){80}}
%\put(30,-5){\line(0,1){100}}
\put(132,90){{\footnotesize $Z_m$}}
\put(190,43){{\footnotesize $m$}}
\put(172,53){{\scriptsize $10$}}
\put(124,53){{\scriptsize $0$}}
\put(64,54){{\scriptsize $-10$}}
%\put(14,44){{\scriptsize $-20$}}
\put(129,80){\line(1,0){2}}
\put(132,78){{\scriptsize $30$}}
\put(129,20){\line(1,0){2}}
\put(131,17){{\scriptsize $-30$}}
%\put(50,103){Determinant-type solution}
%\put(40,103){{\small different}}
%\put(70,99){\line(0,-1){5}}
%\put(100,99){\line(0,-1){5}}
%\put(70,99){\line(1,0){30}}
\put(65,9){\line(0,1){5}}
\put(65,9){\line(-1,0){40}}
\put(20,0){{\small Type $--$}}
\put(72,9){\line(0,1){5}}
\put(78,9){\line(0,1){5}}
\put(83,9){\line(0,1){5}}
\put(100,9){\line(0,1){5}}
\put(72,9){\line(1,0){6}}
\put(83,9){\line(1,0){17}}
\put(70,0){{\small Type $-B$}}
\put(135,9){\line(0,1){5}}
\put(135,9){\line(1,0){30}}
\put(128,0){{\small Type $+B$}}
\put(165,9){\line(0,1){5}}
\put(170,9){\line(0,1){5}}
\put(170,9){\line(1,0){30}}
\put(175,0){{\small Type $++$}}
\put(25,150){\vector(1,0){170}}
\put(130,105){\vector(0,1){90}}
\put(80,105){\line(0,1){90}}
\put(180,105){\line(0,1){90}}
%\put(30,-5){\line(0,1){100}}
\put(132,190){{\footnotesize $Y_m$}}
\put(190,153){{\footnotesize $m$}}
\put(172,153){{\scriptsize $10$}}
\put(124,153){{\scriptsize $0$}}
\put(64,154){{\scriptsize $-10$}}
%\put(14,44){{\scriptsize $-20$}}
\put(129,180){\line(1,0){2}}
\put(132,178){{\scriptsize $30$}}
\put(129,120){\line(1,0){2}}
\put(131,117){{\scriptsize $-30$}}
\put(30,180){\circle{3}}
\put(35,170){\circle*{3}}
\put(40,198){\circle{3}}
\put(45,176){\circle{3}}
\put(50,166){\circle*{3}}
\put(55,194){\circle{3}}
\put(60,172){\circle{3}}
\put(65,162){\circle*{3}}
\put(70,190){\circle{3}}
\put(75,168){\circle*{3}}
\put(80,158){\circle*{3}}
\put(85,186){\circle*{3}}
\put(90,160){\circle{3}}
\put(95,154){\circle*{3}}
\put(100,184){\circle{3}}
\put(105,150){\circle*{3}}
\put(110,152){\circle{3}}
\put(115,174){\circle*{3}}
\put(120,146){\circle*{3}}
\put(125,164){\circle{3}}
\put(130,150){\circle*{3}}
\put(135,142){\circle*{3}}
\put(140,164){\circle{3}}
\put(145,138){\circle*{3}}
\put(150,160){\circle{3}}
\put(155,134){\circle*{3}}
\put(160,156){\circle{3}}
\put(165,130){\circle*{3}}
\put(170,152){\circle{3}}
\put(175,126){\circle*{3}}
\put(180,148){\circle{3}}
\put(185,122){\circle*{3}}
\put(190,144){\circle{3}}
\put(195,118){\circle*{3}}
\put(200,140){\circle{3}}
\put(130,50){\line(5,-2){80}}
\put(130,50){\line(-5,2){100}}
\put(40,87){{\footnotesize $Z_m=mQ$}}
\end{picture}
\end{center}
\caption{The solution associated with Eq.(\ref{eq:pertex3})}
\end{figure}
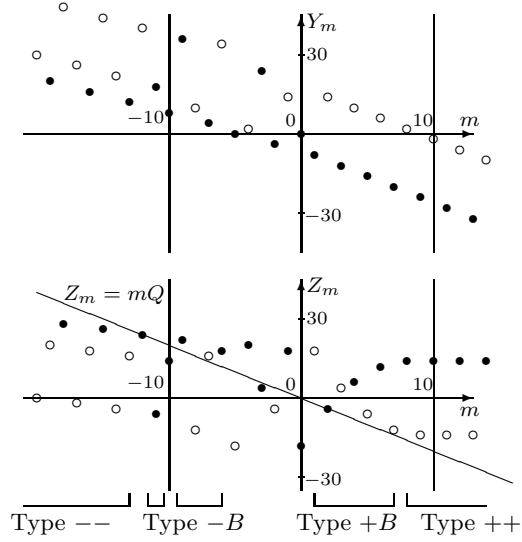

%If the initial value satisfies $ Z_{-12} < 24$, then the solution is written in a slightly different form.
We express the solution such that $24 - Z_{-12} $ is positive and small.
We choose the initial value as  $(\zeta _{-12}, Z_{-12} )= (+1,24 - \varepsilon )$ and $( \eta_{-11} , Y_{-11} )= (+1,18)$ $(0<3 \varepsilon  <1 )$. Then the solution of p-ud PII is written as follows:
\begin{align}
& ( \eta_ m , Y_m ), \: (\zeta _m, Z_m)= \label{eq:pertex4} \\
& \left\{
\begin{array}{lll}
(-1,4j +24 -\varepsilon  ), & ( +1,2j +16 -\varepsilon ) & (m= -3j , \: j \geq 3) \\
(-1,4j + 2 ), & ( -1,2j -14 +\varepsilon ) & (m= -3j +1 , \: j \geq 4) \\
(+1,4j -8 +\varepsilon ), & ( -1,2j +6 ) & (m= -3j +2, \: j \geq 4) \\
(-1,10+\varepsilon ), & (-1, -12+2\varepsilon ) & (m=-8 ) \\
(+1,4+\varepsilon ), & ( -1, 16-\varepsilon ) & (m=-7 ) \\
(-1,34-\varepsilon ), & ( +1, 18) & (m=-6 ) \\
(+1,0+\varepsilon ), & (+1, -18+\varepsilon ) & (m=-5 ) \\
(+1,2), & ( +1, 20-\varepsilon ) & (m=-4 ) \\
(-1,24), & ( -1, 4+\varepsilon ) & (m=-3 ) \\
(+1,-4+\varepsilon ), & (+1, -4-\varepsilon ) & (m=-2 ) \\
(+1,14-2\varepsilon ), & (+1, 18-\varepsilon ) & (m=-1 ) \\
(-1,0+2\varepsilon ), & (-1, -18+3\varepsilon ) & (m=0 ) \\
(+1,-8+\varepsilon ), & (+1, 18-3\varepsilon ) & (m=1 ) \\
(-1,14-\varepsilon ), & (-1, -4+2\varepsilon ) & (m=2 ) \\
(+1,-12+\varepsilon ), & (+1, 4-2\varepsilon ) & (m=3 ) \\
(-1,10-\varepsilon ), & ( -1, 6+\varepsilon ) & (m=4 ) \\
(+1,-16+\varepsilon ), & ( +1, -6-\varepsilon ) & (m=5 ) \\
(-1,6-\varepsilon ), & ( -1, 12) & (m=6 ) \\
(+1,-20+\varepsilon ), & (+1, -12) & (m=7) \\
(-1,2-\varepsilon ), & ( -1, 14-\varepsilon ) & (m=8 ) \\
(+1,-4j -4 +\varepsilon ), & ( +1,-14+\varepsilon ) & (m= 2j -1 , \: j \geq 5) \\
(-1,-4j +18 -\varepsilon ), & ( -1,14-\varepsilon ) & (m= 2j , \: j \geq 5) .
\end{array}
\right. \nonumber
\end{align}

\section{Concluding Remarks} \label{sec:concl} 

In this paper we introduced the simultaneous p-ud PII by setting $y(q^{m+1}) = z(q^{m+1}) z(q^{m}) + 1$ in Eq.(\ref{eq:qPII}) and discussed its solutions.
In particular we investigated the solutions of p-ud PII which are written as ultradiscrete limit of determinant-type solutions of $q$-PII and the solutions whose initial value is perturbed.
We also investigated some patterns of two-parameter solutions.

In the case of differential Painlev\'e equations, the variables of simultaneous equation are chosen to fit with the symplectic structure.
We are not sure that our choice of the variable $y(q^{m+1}) = z(q^{m+1}) z(q^{m}) + 1$ is a good $q$-deformation of the symplectic coordinate, and it would be desirable to find better choice of the variables.
 
One of the merit of studying ultradiscrete equations is that we may obtain exact solutions and it may help analysis of the original $q$-difference equations.
To illustrate problems for future, we show another example of solution of p-ud PII with the parameters $Q=-3$ and $A=7Q$ and the initial condition $( \eta_ 0 , Y_0 ) = (-1,11)$ and $(\zeta _0 , Z_0 ) = (-1,-9) $.
%\begin{align}
%& ( \eta_ m , Y_m ), \: (\zeta _m, Z_m)= \label{eq:pertex5} \\
%& \left\{
%\begin{array}{lll}
%(+1,6j -35 ), & ( +1,3j -38 ) & (m= -3j , \: j \geq 6) \\
%(+1,6j -24 ), & ( +1,3j +14 ) & (m= -3j +1 , \: j \geq 6) \\
%(+1,6j +14 ), & ( +1,3j ) & (m= -3j +2, \: j \geq 6) \\
%(+1,\underline{-5}), & (-1,-18 ) & (m=-15) \\
%(-1,11), & (+1,29) & (m=-14) \\
%(-1,39), & (-1,10 ) & (m=-13) \\
%(+1,\underline{-11}), & (+1,-10 ) & (m=-12) \\
%(+1,16), & (+1,26) & (m=-11) \\
%(+1,22), & (+1,-4 ) & (m=-10) \\
%(+1,\underline{-17}), & (-1,4) & (m=-9) \\
%(-1,27), & (+1,23 ) & (m=-8) \\
%(-1,\underline{-1}), & (-1,-23) & (m=-7) \\
%(+1,\underline{-22}), & (+1,\underline{23} ) & (m=-6) \\
%(-1,37), & (-1,14) & (m=-5) \\
%(+1,\underline{-29}), & (+1,-14 ) & (m=-4) \\
%(+1,6), & (+1,\underline{20}) & (m=-3) \\
%(-1,\underline{-9}), & (-1,-20 ) & (m=-2) \\
%(+1,\underline{-26}), & (+1,\underline{20} ) & (m=-1) \\
%(-1,11), & (-1,-9 ) & (m=0) \\
%(+1,\underline{-41}), & (+1,\underline{9}) & (m=1) \\
%(-1,14), & (-1,\underline{5} ) & (m=2) \\
%(+1,\underline{-47}), & (+1,\underline{-5}) & (m=3) \\
%(-1,8), & (-1,\underline{13} ) & (m=4) \\
%(+1,\underline{-53}), & (+1,\underline{-13}) & (m=5) \\
%(-1,2 ), & ( -1,\underline{15} ) & (m= 6) \\
%(+1,\underline{-6j -41} ), & ( +1,\underline{-15} ) & (m= 2j +1 , \: j \geq 3) \\
%(-1,\underline{-6j +20} ), & ( -1,\underline{15} )  & (m= 2j , \: j \geq 4) 
%\end{array}
%\right.
%\nonumber
%\end{align}
Then the graph of the solution is written as Figure 5, where $\bullet $ (resp. $\circ $) represents the amplitute with $\zeta _m =+1 $ (resp. $\zeta _m =-1 $).
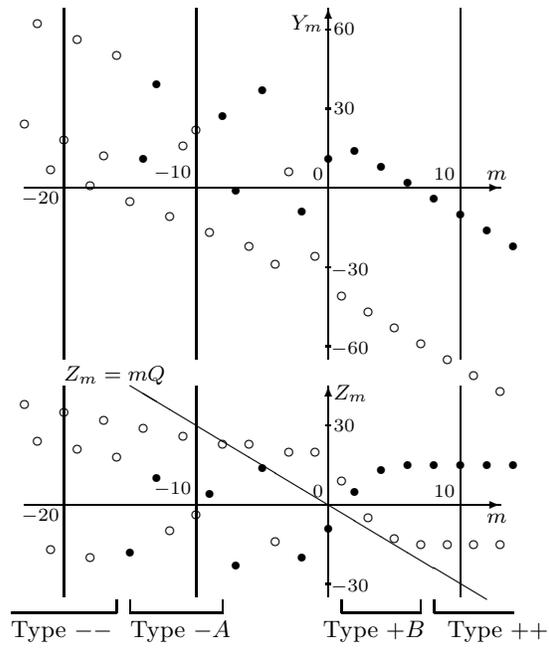
\begin{figure}
\begin{center}
\begin{picture}(230,280)(30,0)
\put(15,88){\circle{3}}
\put(20,74){\circle{3}}
\put(25,33){\circle{3}}
\put(30,85){\circle{3}}
\put(35,71){\circle{3}}
\put(40,30){\circle{3}}
\put(45,82){\circle{3}}
\put(50,68){\circle{3}}
\put(55,32){\circle*{3}}%m=-15
\put(60,79){\circle{3}}
\put(65,60){\circle*{3}}
\put(70,40){\circle{3}}
\put(75,76){\circle{3}}
\put(80,46){\circle{3}}%m=-10
\put(85,54){\circle*{3}}
\put(90,73){\circle{3}}
\put(95,27){\circle*{3}}
\put(100,73){\circle{3}}
\put(105,64){\circle*{3}}
\put(110,36){\circle{3}}
\put(115,70){\circle{3}}
\put(120,30){\circle*{3}}
\put(125,70){\circle{3}}
\put(130,41){\circle*{3}}%m=0
\put(135,59){\circle{3}}
\put(140,55){\circle*{3}}
\put(145,45){\circle{3}}
\put(150,63){\circle*{3}}
\put(155,37){\circle{3}}
\put(160,65){\circle*{3}}
\put(165,35){\circle{3}}
\put(170,65){\circle*{3}}
\put(175,35){\circle{3}}
\put(180,65){\circle*{3}}%m=10
\put(185,35){\circle{3}}
\put(190,65){\circle*{3}}
\put(195,35){\circle{3}}
\put(200,65){\circle*{3}}
\put(15,50){\vector(1,0){180}}
\put(130,15){\vector(0,1){80}}
\put(80,15){\line(0,1){80}}
\put(180,15){\line(0,1){80}}
\put(30,15){\line(0,1){80}}
\put(132,90){{\footnotesize $Z_m$}}
\put(190,43){{\footnotesize $m$}}
\put(170,53){{\scriptsize $10$}}
\put(124,53){{\scriptsize $0$}}
\put(64,54){{\scriptsize $-10$}}
\put(14,44){{\scriptsize $-20$}}
\put(129,80){\line(1,0){2}}
\put(132,78){{\scriptsize $30$}}
\put(129,20){\line(1,0){2}}
\put(131,17){{\scriptsize $-30$}}
%\put(50,103){Determinant-type solution}
%\put(40,103){{\small different}}
%\put(70,99){\line(0,-1){5}}
%\put(100,99){\line(0,-1){5}}
%\put(70,99){\line(1,0){30}}
\put(50,9){\line(0,1){5}}
\put(50,9){\line(-1,0){40}}
\put(10,0){{\small Type $--$}}
\put(55,9){\line(0,1){5}}
\put(90,9){\line(0,1){5}}
\put(55,9){\line(1,0){35}}
\put(55,0){{\small Type $-A$}}
\put(135,9){\line(0,1){5}}
\put(135,9){\line(1,0){30}}
\put(128,0){{\small Type $+B$}}
\put(165,9){\line(0,1){5}}
\put(170,9){\line(0,1){5}}
\put(170,9){\line(1,0){30}}
\put(175,0){{\small Type $++$}}
\put(15,170){\vector(1,0){180}}
\put(130,105){\vector(0,1){133}}
\put(180,105){\line(0,1){133}}
\put(80,105){\line(0,1){133}}
\put(30,105){\line(0,1){133}}
\put(116,230){{\footnotesize $Y_m$}}
\put(190,173){{\footnotesize $m$}}
\put(170,173){{\scriptsize $10$}}
\put(124,173){{\scriptsize $0$}}
\put(64,174){{\scriptsize $-10$}}
\put(14,164){{\scriptsize $-20$}}
\put(129,230){\line(1,0){2}}
\put(129,200){\line(1,0){2}}
\put(132,228){{\scriptsize $60$}}
\put(132,198){{\scriptsize $30$}}
\put(129,140){\line(1,0){2}}
\put(131,137){{\scriptsize $-30$}}
\put(129,110){\line(1,0){2}}
\put(131,107){{\scriptsize $-60$}}
\put(15,194){\circle{3}}
\put(20,232){\circle{3}}
\put(25,177){\circle{3}}
\put(30,188){\circle{3}}
\put(35,226){\circle{3}}
\put(40,171){\circle{3}}
\put(45,182){\circle{3}}
\put(50,220){\circle{3}}
\put(55,165){\circle{3}}
\put(60,181){\circle*{3}}
\put(65,209){\circle*{3}}
\put(70,159){\circle{3}}
\put(75,186){\circle{3}}
\put(80,192){\circle{3}}%m=10
\put(85,153){\circle{3}}
\put(90,197){\circle*{3}}
\put(95,169){\circle*{3}}
\put(100,148){\circle{3}}
\put(105,207){\circle*{3}}
\put(110,141){\circle{3}}
\put(115,176){\circle{3}}
\put(120,161){\circle*{3}}
\put(125,144){\circle{3}}
\put(130,181){\circle*{3}}%m=0
\put(135,129){\circle{3}}
\put(140,184){\circle*{3}}
\put(145,123){\circle{3}}
\put(150,178){\circle*{3}}
\put(155,117){\circle{3}}
\put(160,172){\circle*{3}}
\put(165,111){\circle{3}}
\put(170,166){\circle*{3}}
\put(175,105){\circle{3}}
\put(180,160){\circle*{3}}%m=10
\put(185,99){\circle{3}}
\put(190,154){\circle*{3}}
\put(195,93){\circle{3}}
\put(200,148){\circle*{3}}
\put(130,50){\line(5,-3){60}}
\put(130,50){\line(-5,3){75}}
\put(30,97){{\footnotesize $Z_m=mQ$}}
\end{picture}
\end{center}
\caption{The solution with the condition $( \eta_ 0 , Y_0 ) = (-1,11) \: \atop{(\zeta _0 , Z_0 ) = (-1,-9)}$}
\end{figure}

If $m \leq -16 $, then the solution is in the form of Proposition \ref{prop:sol<<} (i.e.~of Type $--$), and it is written as 
\begin{align}
& ( \eta_ m , Y_m ), \: (\zeta _m, Z_m)= \label{eq:m-inf} \\
& \left\{
\begin{array}{lll}
(+1,6j -35 ), & ( +1,3j -38 ) & (m= -3j , \: j \geq 6) \\
(+1,6j -24 ), & ( +1,3j +14 ) & (m= -3j +1 , \: j \geq 6) \\
(+1,6j +14 ), & ( +1,3j ) & (m= -3j +2, \: j \geq 6) .
\end{array}
\right.
\nonumber
\end{align}
In the case $-15\leq m\leq -8$ (resp. $1 \leq m\leq 7 $), the solution is of Type $-A$ (resp.~of Type $+B$).
If $m\geq 7$, then the solution is written in the form of Proposition \ref{prop:sol>>} (i.e.~of Type $++$), and it is written as 
\begin{align}
& ( \eta_ m , Y_m ), \: (\zeta _m, Z_m)= \label{eq:m+inf} \\
& \left\{
\begin{array}{lll}
(-1,-6j +20 ), & ( -1,15 ) & (m= 2j , \: j \geq 3) \\
(+1,-6j -41 ), & ( +1,-15 ) & (m= 2j +1 , \: j \geq 3) .
\end{array}
\right.
\nonumber
\end{align}
It is apparent from Figure 5 that the tendency that the inequality $Y_m <0 $ (below the $m$-axis) or $mQ -Z_{m} <0$ (above the line $Z_m=mQ$) holds increases as the value $m$ increases, and it is also apparent for solutions in Eqs.(\ref{eq:pertex1}), (\ref{eq:pertex2}), (\ref{eq:pertex3}) and (\ref{eq:pertex4}) (see Figures 3 and 4).

%%written in the form of Proposition \ref{prop:sol<<} (resp. Proposition \ref{prop:sol>>}).
%If $m \leq -16 $ (resp. $m\geq 7$), then the solution is written in the form of Type $--$ (resp.~of Type $++$).
%%written in the form of Proposition \ref{prop:sol<<} (resp. Proposition \ref{prop:sol>>}).
%In the case $-15\leq m\leq -8$ (resp. $1 \leq m\leq 7 $), the solution is of Type $-A$ (resp.~of Type $+B$).
%%written in the form of Eqs.(\ref{eq:condZ---Y-++}) and (\ref{eq:solZ---Y-++}) (resp. Eqs.(\ref{eq:condZ++Y+-}) and (\ref{eq:solZ++Y+-})).
%We underline the values such that the inequality $Y_m <0 $ or $mQ -Z_{m} <0$ holds.
%It is apparent that the tendency that the inequality $Y_m <0 $ or $mQ -Z_{m} <0$ holds increases as the value $m$ increases, and it is also apparent for solutions in Eqs.(\ref{eq:pertex1}), (\ref{eq:pertex2}), (\ref{eq:pertex3}) and (\ref{eq:pertex4}).

%%We conjecture that, if a solution of p-ud PII does not have the forward indefinite evolution nor the backward indefinite evolution for all $m \in \mathbb{Z}$, then there exist some values $m'$ and $m''$ such that the solution is written in the form of Proposition \ref{prop:sol<<} for $m<m'$ and in the form of Proposition \ref{prop:sol>>} for $m>m''$.
We conjecture that, if a solution of p-ud PII does not have forward indefinite evolution nor backward indefinite evolution for all $m \in \mathbb{Z}$, then there exist some values $m'$ and $m''$ such that the solution is written in the form of Type $--$ for $m<m'$ and in the form of Type $++$ for $m>m''$.
Then we propose a problem that how the asymptotics as $m \to -\infty $ is concerned with the asymptotics as $m \to +\infty $. 
In the example given in Figure 5, the asymptotics as $m \to -\infty $ given in Eq.(\ref{eq:m-inf}) is connected with the asymptotics given in Eq.(\ref{eq:m+inf}) as $m \to +\infty $.

In order to study solutions of $q$-Painlev\'e equations of other types, ultradiscretization with parity variables of them and analysis of the ultradiscrete solutions should be performed.
Note that p-ud Painlev\'e III (resp. p-ud Painlev\'e VI) was derived in \cite{Iso} (resp. \cite{TT}).

\section*{Acknowledgment}
The authors thank Shin Isojima for discussions and valuable comments.
%The authors also thank the referees for valuable comments.
The second author was supported by JSPS KAKENHI Grant Number JP26400122 and by Chuo University Overseas Research Program.

\appendix
\section{Summary of ultradiscrete limit of determinant-type solutions in \cite{Iga}.}
It is shown in \cite{HKW} that q-PII (Eq.(\ref{eq:qPII})) with $a=q^{2N+1}$ $(N\in\mathbb{Z})$ admits a class of special solutions. 
Eq.(\ref{eq:qPII}) with $a=q^{2N+1}$ is solved by
\begin{align}
z^{(N)}(q^m) &=
\begin{cases}
\dfrac{g^{(N)}(m) g^{(N+1)}(m+1)}{q^N g^{(N)}(m+1) g^{(N+1)}(m)} & (N\geq0), \\
\dfrac{g^{(N)}(m) g^{(N+1)}(m+1)}{q^{N+1} g^{(N)}(m+1) g^{(N+1)}(m)} & (N<0), \\
\end{cases}\label{eq:gtozm} 
\end{align}
\begin{align}
& g^{(N)}(m) =  \label{eq:g_detm}\\
& \begin{cases}
\begin{vmatrix}
w(m) & w(m+2) & \cdots & w(m+2N-2) \\
w(m-1) & w(m+1) & \cdots & w(m+2N-3) \\
\vdots & \vdots & \ddots & \vdots \\
w(m-N+1) & w(m-N+3) & \cdots & w(m+N-1)
\end{vmatrix} & (N>0), \\
1 & (N=0), \\
\begin{vmatrix}
w(m-1) & w(m-3) & \cdots & w(m-2|N|+1) \\
w(m) & w(m-2) & \cdots & w(m-2|N|+2) \\
\vdots & \vdots & \ddots & \vdots \\
w(m+|N|-2) & w(m+|N|-4) & \cdots & w(m-|N|) \\
\end{vmatrix} & (N<0), \\
\end{cases} \nonumber
\end{align}
where $w(m)$ is a solution of $q$-Airy equation
\begin{equation}
w(m+1) - q^m w(m) + w(m-1) = 0 . \label{eq:qAiry}
\end{equation}

We review the ultradiscrete limit of the determinant-type solutions $z^{(M)}(q^m)$ for $M \in\mathbb{Z} _{<0}$ by following section 4 of \cite{Iga}. 
Recall that the $q$-Airy equation has special solutions, the $q$-Ai function $a(m)$ and the $q$-Bi function $b(m)$ (see \cite{IIT} for the expression), and the general solution is given by the linear combination. We express a solution of the $q$-Airy equation as
\begin{equation} \label{eq:w}
w(m) = \alpha e^{A'/\varepsilon} a(m) + \beta e^{B'/\varepsilon} b(m),
\end{equation}
where $\alpha , \beta \in \{ \pm 1 \}$ and $A',B' \in \mathbb{R}$.

Let $M \in \mathbb{Z}_{<0} $ and set $N=-M$.
The p-ud analogue of $g^{(N)}(m)$ for $N \in \mathbb{Z}_{\geq 0}  $ was calculated in \cite[Proposition 3]{IIT}.
On the other hand, it follows from Eq.(\ref{eq:g_detm}) that
\begin{equation}
g^{(M)}(m) = g^{(N)}(m-N) . \label{eq:NtoM}
\end{equation}
By combining them, we have
\begin{prop}\label{prop:gpu_M}
For $m \leq M+1 $, the p-ud analogue of $g^{(M)}(m)$ is given by
\begin{align}
& (\gamma^{(M)}(m), G^{(M)}(m)) = \label{eq:gpu_M}\\
& \begin{cases}
(\gamma_0^{(M)}(m), G_0^{(M)}(m)) & (B'-A' < f_1^{(M)}(m)) \\
(\gamma_k^{(M)}(m), G_k^{(M)}(m)) & (f_k^{(M)}(m) \le B'-A' < f_{k+1}^{(M)}(m)) \\
(\gamma_{-M}^{(M)}(m), G_{-M}^{(M)}(m)) & (f_{-M}^{(M)}(m) \le B'-A' ),
\end{cases} \nonumber
\end{align}
where
\begin{align}
\gamma_k^{(M)}(m) =& (-1)^{Mk-k(k+1)/2} \alpha^{M-k} \beta^k, \label{eq:gamma_k_M} \\
G_k^{(M)}(m)
=& (-M-k)A' + kB' + \biggl[ \left(-k - \frac12 M \right)m^2 \nonumber \\
& \! \! \! \! + \left\{-3k^2 + (-4M+1)k - \frac{M(M-2)}{2} \right\}m - \frac83k^3  \nonumber \\
& \! \! \! \! \! \! \! \! \! \! \! \! + \left(-5M+\frac32 \right)k^2 + \left(-3M^2+2M+\frac16\right)k - \frac{M(M-1)(M-2)}{6} \biggr]Q, \nonumber \\
%\label{eq:G_k_M} \\
f_k^{(M)}(m) =& G_{k-1}^{(M)}(m) - G_k^{(M)}(m) - A' + B' \nonumber \\
=& \left\{ (m+2M+3k-2)^2 - (M+k)(M+k-1) \right\} Q, \nonumber 
%\label{eq:f_k_M}
\end{align}
$k = 0, 1, \ldots, -M$.
\end{prop}

We may write the p-ud analogue of \eqref{eq:gtozm} as follows:
\begin{align}
& \zeta^{(M)}(m) = \gamma^{(M+1)}(m+1) \gamma^{(M+1)}(m) \gamma^{(M)}(m+1) \gamma^{(m)}(m), \label{eq:GtoZ_M} \\
& Z^{(M)}(m) = G^{(M+1)}(m+1) - G^{(M+1)}(m) - G^{(M)}(m+1) \nonumber \\
& \qquad \qquad \quad + G^{(M)}(m) - (M+1)Q . \nonumber
\end{align}
For $m\le M+1$, we set
\begin{align}
& h_{{\rm I},l}^{(M)} (m) = f_l^{(M)}(m), \ h_{{\rm II},l}^{(M)} (m) = f_l^{(M)}(m+1), \nonumber \\
& h_{{\rm III},l}^{(M)} (m) = f_l^{(M+1)}(m), \ h_{{\rm IV},l}^{(M)} (m) = f_l^{(M+1)}(m+1),
\end{align}
where $f_l^{(M)}(m)$ was defined in \eqref{eq:gamma_k_M}.
%\eqref{eq:f_k_M}.
Then we have
\begin{equation}
h_{{\rm I},l}^{(M)} (m) < h_{{\rm II},l}^{(M)} (m) < h_{{\rm III},l}^{(M)} (m) < h_{{\rm IV},l}^{(M)} (m) = h_{{\rm I},l+1} ^{(M)}(m) 
\label{eq:f_ineq_M}
\end{equation}
for $l = 1, \ldots, -M$.
%\end{lem}
\begin{prop}\label{prop:zpu_M}
For $m\le M+1$, we have
\begin{align}
& (\zeta^{(M)}(m), Z^{(M)}(m))= \\
& 
\begin{cases}
(+1, (-m-2M-1)Q) & (B'-A' < h_{{\rm I},1}^{(M)}(m) \\
(\zeta_{{\rm I},l}^{(M)}(m), Z_{{\rm I},l}^{(M)}(m)) & (h_{{\rm I},l}^{(M)}(m) \le B'-A' < h_{{\rm II},l}^{(M)}(m)) \\
(\zeta_{{\rm II},l}^{(M)}(m), Z_{{\rm II},l}^{(M)}(m)) & (h_{{\rm II},l}^{(M)}(m) \le B'-A' < h_{{\rm III},l}^{(M)}(m)) \\
(\zeta_{{\rm III},l}^{(M)}(m), Z_{{\rm III},l}^{(M)}(m)) & (h_{{\rm III},l}^{(M)}(m) \le B'-A' < h_{{\rm I},l+1}^{(M)}(m)) \\
(+1, mQ) & (h_{{\rm II},-M}^{(M)}(m) \le B'-A'), \\
\end{cases} \nonumber
\end{align}
where
\begin{align}
&\begin{cases}
\zeta_{{\rm I},l}^{(M)}(m) = (-1)^{M+l} \alpha \beta, \\
Z_{{\rm I},l}^{(M)}(m) = B'-A' - \{m^2+(6l+4M-3)m \\
\qquad \qquad \qquad + 8l^2+(10M-7)l+3M^2-5M+1\}Q,
\end{cases}\\
&\begin{cases}
\zeta_{{\rm II},l}^{(M)}(m) = +1, \\
Z_{{\rm II},l}^{(M)}(m) = (m+2M+2l)Q,
\end{cases} \nonumber \\
&\begin{cases}
\zeta_{{\rm III},l}^{(M)}(m) = (-1)^{M+l+1} \alpha \beta, \\
Z_{{\rm III},l}^{(M)}(m) = A' - B' + \{m^2+(6l+4M+1)m \\
\qquad \qquad \qquad + 8l^2+(10M+1)l+3M^2+M\}Q,
\end{cases}
\nonumber 
\end{align}
for $l=1, 2, \ldots, -M$.
\end{prop}
We assume that the values of $A',B'$ and $Q$ are chosen as $m_0$ satisfies $m_0 \leq \min [ 3M+1, -M(M+1)/2 -1 ]$ and $m_0^2Q \le B'-A'<(m_0+1)^2Q$.
Set $P_{0} = (m_0 + 1)^2Q$ and $P_j = \{m_0^2 - (M+j)(M+j-1)\}Q $ $(j = 1, \ldots, -M)$.
Then there exists an integer $k_0 \in \{ 0 , 1 , \cdots ,-M-1 \}$ such that $P_{k_0 +1} \le B'-A' < P_{k_0}$ holds.
Using these notation, the result is written as follows:
\begin{thm} \label{thm:zpum_0_M} $($c.f. \cite[Theorem 5]{Iga}$)$
Assume that we have $m_0$ and $k_0 $ mentioned above for assigned values of $A',B'$ and $Q$.
Then the following function $(\zeta^{(M)}(m), Z^{(M)}(m))$ is obtained by the p-ud limit of the solution of $q$-PII in terms of determinants.\\
(I) If $m \leq m_0 + M$, then 
\begin{equation}
(\zeta^{(M)}(m), Z^{(M)}(m)) = (+1, mQ).
\end{equation}
(II)  If $m_0 + M + 1 \leq m \leq m_0-2M-3k_0$ $(k_0\neq 0)$ or $m_0 + M + 1 \leq m \leq m_0-2M-1$ $(k_0=0)$, then
\begin{align}
& (\zeta^{(M)}(m), Z^{(M)}(m)) = \nonumber \\
& \begin{cases}
(+1, (m_0+M+j)Q) & (m=m_0 + M + 3j - 2) \\
((-1)^{j+1} \alpha \beta, B'-A'-(m_0^2+m_0-j^2-M)Q) & (m=m_0 + M + 3j -1) \\
((-1)^{j+1} \alpha \beta, A'-B'+(m_0^2+m_0-j^2+2j+M)Q) & (m=m_0 + M + 3j),
\end{cases}
\end{align}
where $1 \leq j \leq -M-k_0$.\\
(III) If $m_0 - 2M - 3k_0 + 1 \leq m \leq m_0 - 2M - 1$ and $k_0 \neq 0$, then
\begin{align}
& (\zeta^{(M)}(m), Z^{(M)}(m)) = \nonumber \\
&  \begin{cases}
((-1)^{j+1} \alpha \beta, B'-A'-(m_0^2-m_0-j^2-M)Q) & (m=m_0+M+3j-2) \\
((-1)^{j+1} \alpha \beta, A'-B'+(m_0^2-m_0-j^2+2j+M)Q) & (m=m_0+M+3j-1)  \\
(+1, (m_0+M+j)Q) & (m=m_0+M+3j),
\end{cases}
\end{align}
where $-M-k_0+1 \leq j \leq -M$ for the first case and $-M-k_0+1 \leq j \leq -M-1$ for the second and the third cases.\\
(IV) If $m_0 - 2M -1 \leq m \leq M+1$ $(k_0\neq 0)$ or $m_0 - 2M \leq m \leq M+1$ $(k_0=0)$, then
\begin{equation}
(\zeta^{(M)}(m), Z^{(M)}(m)) = (+1, (-m-2M-1)Q).
\end{equation}
\end{thm}
By setting $C= B'-A'-(m_0^2+m_0)Q $ , we obtain Proposition \ref{prop:zpum_0_M}.

\end{document}